\documentclass{article}[12pts]

\usepackage{overpic}
\usepackage{amsmath}
\usepackage{xfrac}
\usepackage{blindtext}
\usepackage{comment}
\usepackage{faktor}
\usepackage[all]{xy}
\usepackage[french, main=english]{babel}
\usepackage[utf8x]{inputenc}
\usepackage{float}
\usepackage{graphicx}
\usepackage[colorinlistoftodos]{todonotes}
\usepackage{url}
\usepackage{hyperref}
\usepackage{array}
\usepackage{tabularx}
\usepackage{setspace}
\usepackage[T1]{fontenc}
\usepackage{subfig}
\usepackage[margin=1in]{geometry} 
\usepackage{fancyhdr}
\usepackage{rotating}
\usepackage{amsfonts,amssymb,amsthm,scrextend}
\usepackage{pinlabel}
\usepackage[nottoc,numbib]{tocbibind}
\usepackage{appendix}
\usepackage{lscape}

\newcommand{\N}{\mathbb{N}}

\newcommand{\Z}{\mathbb{Z}}

\newcommand{\R}{\mathbb{R}}

\newcommand{\F}{\mathcal{F}}

\newcommand{\deffct}[5]{#1: \begin{array}{ccc}
    #2 & \to & #3  \\
     #4 & \mapsto & \displaystyle #5
\end{array}}
\newcommand{\trajb}[1]{\overline{\mathcal{L}}(#1)}

\newcommand{\wb}[2]{\overline{W}^{#1}(#2)}
\newcommand{\wbi}[3]{\overline{W}^{#1}_{#2}(#3)}

\newcommand{\fibration}{E \overset{\pi}{\to} X}

\newcommand{\Crit}{\textup{Crit}}
\newcommand{\Or}{\textup{Or}}
\newcommand{\Id}{\textup{Id}}

\newcommand{\ev}{\textup{ev}}

\newcommand{\blue}[1]{\textcolor{blue}{#1}}

\newcommand{\ftimes}[2]{\hspace{0.1em} _{#1}\!\times_{#2}}

\newcommand{\ls}[1]{\mathcal{L} #1}
\newcommand{\CS}{\textup{CS}_{DG}}

\newcommand{\m}{\mathbf{\tilde{m}}}

\newcommand{\COY}{C_*(\Omega Y)}
\newcommand{\COYi}[1]{C_*(\Omega Y^{#1})}
\newcommand{\PP}{\textup{PP}_{DG}}
\newcommand{\functraj}[2]{\overline{\mathcal{M}}^{#1}(#2)}

\renewcommand{\qed}{\hfill$\blacksquare$}

\renewenvironment{proof}{\begin{addmargin}[1em]{0em}\begin{newproof}}{\end{newproof}\end{addmargin}\qed}

\definecolor{darkgreen}{RGB}{0,168,0}

\newtheorem{thm}{Theorem}[section]

\newtheorem{defi}[thm]{Definition}
\newtheorem{rem}[thm]{Remark}
\newtheorem{prop}[thm]{Proposition}
\newtheorem{cor}[thm]{Corollary}
\newtheorem{lemme}[thm]{Lemma}

\newtheorem{theorem}{Theorem}
\newtheorem{proposition}[theorem]{Proposition}

\author{Robin Riegel \footnote{email : r.riegel@math.unistra.fr} \\
IRMA, Université de Strasbourg}
\date{}

\title{The Path-product in Morse Homology with differential graded coefficients}

\begin{document}
\selectlanguage{english}
\renewcommand{\contentsname}{Table of Contents}
\renewcommand{\proofname}{Proof}
\renewcommand{\refname}{References}
\renewcommand{\figurename}{Figure}
\setlength{\parindent}{0pt}
\maketitle

\begin{center}

    \textbf{Abstract} \\

\end{center}

We will use the tools developed in \cite{Rie24} to give a Morse-theoretical description of a string topology product on the homology of the space of paths in a manifold $Y$ with endpoints in a submanifold $X$ and a module structure on this homology over the Chas-Sullivan ring of $Y$. These operations have been defined and studied by Stegemeyer in \cite{Max23} in singular homology.

\tableofcontents
\newpage

\section{Introduction}

Let $(Y^k,\star_Y)$ be a pointed, oriented, closed, and connected manifold and let $(X^n,\star_X)$ be a pointed, oriented, closed, and connected submanifold of $Y$.
The goal of this article is to use Morse Homology with differential graded coefficients defined and studied by Barraud-Damian-Humilière-Oancea in \cite{BDHO23} and the tools developed in \cite{Rie24} in order to give a Morse model for the \textbf{Path-product} $$\Lambda :H_i(\mathcal{P}_{X} Y) \otimes H_j(\mathcal{P}_{X} Y) \to H_{i+j-n}(\mathcal{P}_{X} Y) $$ defined by Stegemeyer in \cite{Max23}, where $$\mathcal{P}_X Y = \{ \gamma : [0,a] \to Y \textup{ continuous }, \ \gamma(0),\gamma(a) \in X\}.$$ We start by a brief introduction of the main tools that we will use in this article.

\subsection{Context}

The authors of \cite{BDHO23} construct \textbf{Morse complex of $X$ with coefficients in a differential graded (DG) module} $(\F_*, \partial_{\F})$ over the differential graded algebra (DGA) $R_* = C_*(\Omega X)$ of the cubical complex of the space of loops in $X$ based at $\star$. This is also referred to as an \textbf{enriched Morse complex}. Given a Morse function $f : X \to \R$ and $\xi$, a pseudo-gradient adapted to $f$, this complex writes 

$$C_*(X,\F_*) = \F_* \otimes \Z \Crit(f),$$

where the differential is twisted by a family of chains called, in reference to the seminal paper \cite{BC07}, \textbf{Barraud-Cornea twisting cocycle}

$$\left\{ m_{x,y} \in C_{|x|-|y|-1}(\Omega X), x,y \in \Crit(f) \right\}.$$

The cocycle $(m_{x,y})_{x,y}$ is obtained by evaluating in $\Omega X$ representatives of the fundamental classes of the moduli spaces of Morse trajectories $\trajb{x,y}$. The twisted differential is given by

$$\partial (\alpha \otimes x) = \partial_{\F} \alpha \otimes x + (-1)^{|\alpha|} \sum_y \alpha \cdot m_{x,y} \otimes y.$$

In \cite{Rie24}, we make use of their construction and, in particular, of one of their main results, the Fibration Theorem, in order to build a Morse theoretic Chas-Sullivan product $$\CS : H_*(E)^{\otimes 2} \to H_*(E),$$ for any (Hurewicz) fibration $\fibration$ endowed with a morphism of fibrations $m : E \ftimes{\pi}{\pi} E \to E$.

Given a fibration $E \overset{\pi}{\to} X$, a \textbf{transitive lifting function} $\Phi : E \ftimes{\pi}{\ev_0} \mathcal{P}X \to E$ is a map that lifts all paths in $X$ and respects concatenation in the sense that lifting a concatenation of two paths is the same as lifting them one after the other : $\Phi(\Phi(e,\gamma),\delta) = \Phi(e,\gamma \#\delta).$
Let $F = \pi^{-1}(\star_X)$. The restriction $\Phi : F \times \Omega X \to F$ induces a $C_*(\Omega X)$-module structure on the cubical complex $C_*(F)$.\\

\textbf{Fibration Theorem} (\cite[Theorem 7.2.1]{BDHO23})
    \emph{Let $\fibration$ be a fibration with model fiber $F=\pi^{-1}(\star)$ and equip $C_*(F)$ with the $C_*(\Omega X)$-module structure induced by a transitive lifting function $\Phi : E \ftimes{\pi}{\ev_0} \mathcal{P}X \to E$ associated with this fibration. Then, there exists a quasi-isomorphism denoted $$\Psi_E : C_*(X, C_*(F)) \to C_*(E).$$}

\paragraph{DG Chas-Sullivan product}\label{CSDG}(\cite[Theorem 7.1]{Rie24})
\emph{Let $\fibration$ be a fibration endowed with a morphism of fibrations $m : E \ftimes{\pi}{\pi} E \to E$. Let $F = \pi^{-1}(\star)$ and endow $\F = C_*(F)$ with a DG right $C_*(\Omega X)$-module structure induced by a transitive lifting function. There exists a degree $-n$ product
$$\CS : H_*(X, \F)^{\otimes 2} \to H_*(X, \F)$$}

\emph{that is associative if $m_*$ is associative in homology, graded commutative if $m_*$ is commutative in homology, admits a neutral element if there exists a unit section $s : X \to E$ and this product satisfies the Functoriality and Spectral sequence properties.}

\emph{Moreover, this product corresponds in homology, via the Fibration Theorem, to the product $\mu_* : H_i(E) \otimes H_j(E) \to H_{i+j-n}(E)$ defined by Gruher-Salvatore in \cite{GS07}. In particular, if the fibration is the loop-loop fibration $\Omega X \hookrightarrow \mathcal{L}X \to X$, then $\CS$ corresponds to the Chas-Sullivan product.}

\vspace{0.5em}

The tools and techniques used to build such a product also enable to build a Morse model $$ \PP : H_i(X \times X, C_*(\Omega Y)) \otimes H_j(X \times X, C_*(\Omega Y)) \to H_{i+j-n}(X \times X, C_*(\Omega Y)) $$ for the Path-product $$\Lambda :H_i(\mathcal{P}_{X} Y) \otimes H_j(\mathcal{P}_{X} Y) \to H_{i+j-n}(\mathcal{P}_{X} Y) $$ and a module structure $$\cdot_M : H_i(Y,C_*(\Omega Y)) \otimes H_j(X \times X, C_*(\Omega Y)) \to H_{i+j-k}(X \times X, C_*(\Omega Y)),$$ where  $H_*(Y,C_*(\Omega Y))$ is endowed with the ring structure $\CS.$

\subsection{The Path-product}

Let $Y$ be a topological space, let $X^n$ be a pointed, oriented, closed, and connected manifold and let $f : X \to Y$ be a continuous map. Let $\mathcal{P}Y = \{ \gamma : [0,a] \to Y, \textup{ continuous}\}$ be the space of Moore paths. Consider the pullback fibration $$\Omega Y \hookrightarrow \mathcal{P}_{X,f} Y \overset{(\ev_0, \ev_1)}{\to} X^2$$ of the fibration  $ \Omega Y \hookrightarrow \mathcal{P} Y\overset{(\ev_0, \ev_1)}{\to} Y^2$
 by $f \times f : X^2 \to Y^2.$ The Fibration Theorem \cite[Theorem 7.2.1]{BDHO23} implies that $$H_*(X^2, (f \times f)^*C_*(\Omega Y)) \cong H_*(\mathcal{P}_{X,f} Y).$$ For readability, we will often denote $H_*(X^2, \COY ;f)$ this homology or $H_*(X^2, \COY)$ when $f$ is implied. We will use similar notations for the complexes of chains.
The construction by Stegemeyer in \cite{Max23} of a degree $-n$ product
$$\Lambda :H_*(\mathcal{P}_{X,f} Y)^{\otimes 2} \to H_*(\mathcal{P}_{X,f} Y) $$

relies on the idea of "intersecting on the base" to obtain concatenable paths and concatenating. This idea is similar to the construction of the Chas-Sullivan product but does not fall under the category of products defined in \cite[Theorem 7.1]{Rie24} since the intersection does not occur along the diagonal $\Delta : X^2 \to X^4$ but rather along $D : X^3 \to X^4$, $D(a,b,c) = (a,b,b,c)$. In other words, the concatenation of interest is $$m : \mathcal{P}_{X,f} Y \ftimes{\ev_1}{\ev_0} \mathcal{P}_{X,f} Y \to  \mathcal{P}_{X,f} Y$$
and not  $$\mathcal{P}_{X,f} Y \ftimes{(\ev_0,\ev_1)}{(\ev_0,\ev_1)} \mathcal{P}_{X,f} Y \to  \mathcal{P}_{X,f} Y.$$
The concatenation $m$ is not a morphism of fibrations in the sense defined in \cite[Section 5]{Rie24} since  $\mathcal{P}_{X,f} Y\ftimes{\ev_1}{\ev_0} \mathcal{P}_{X,f} Y$ is seen as the total space of a fibration over $X^3$ and $\mathcal{P}_{X,f} Y$ is seen as the total space of a fibration over $X^2$. Therefore, in order to define a Morse equivalent of this map, we generalize \cite[Theorem 5.8]{Rie24} for a broader class of morphisms that we define, \emph{i.e} \textbf{relative morphism of fibrations}.

\begin{defi}\label{defi : relative morphism of fibrations}
    Let $Z_0$ and $Z_1$ be topological spaces and let $E_0 \overset{\pi_0}{\to} Z_0$ and $E_1 \overset{\pi_1}{\to} Z_1$ be fibrations. Let $g : Z_0 \to Z_1$ be a continuous map. We say that $\varphi : E_0 \to E_1$ is a \textbf{morphism of fibrations relative to $g$} if $\pi_1 \circ\varphi = g \circ \pi_0.$
$$ \xymatrix{
    E_0 \ar[r]^{\varphi} \ar[d]_{\pi_0} & E_1 \ar[d]^{\pi_1} \\
    Z_0 \ar[r]_{g} & Z_1.
    }$$
\end{defi}

The following result generalizes \cite[Theorem 5.8]{Rie24} using \cite[Proposition 9.8.1]{BDHO23}.

\begin{theorem}[Theorem \ref{cor : morphism of fibrations relative to ...}]\label{theorem C'}
    Let $Z_0$ and $Z_1$ be two pointed, closed, and connected manifolds. Let $F_0 \hookrightarrow E_0 \to Z_0$ and $F_1 \hookrightarrow E_1 \to Z_1$ be two fibrations. If $\varphi: E_0 \to E_1$ is a morphism of fibrations relative to a continuous map $g : Z_0 \to Z_1$, then there exists a morphism of complexes $\tilde{\varphi} : C_*(Z_0, C_*(F_0)) \to C_*(Z_0, g^*C_*(F_1))$ such that the following diagram commutes up to chain homotopy

    \[
    \xymatrix{
    C_*(Z_0, C_*(F_0)) \ar[r]^-{g_* \circ \Tilde{\varphi}} \ar[d]^{\Psi_0} & C_*(Z_1, C_*(F_1)) \ar[d]_{\Psi_1} \\
    C_*(E_0) \ar[r]_{\varphi_*} & C_*(E_1),
    }
    \] where the morphisms $\Psi_0$ and $\Psi_1$ are the quasi-isomorphisms given by the Fibration Theorem \cite[Theorem 7.2.1]{BDHO23}.
\end{theorem}

The concatenation $m : \mathcal{P}_{X,f} Y \ftimes{\ev_1}{\ev_0} \mathcal{P}_{X,f} Y \to \mathcal{P}_{X,f} Y$ is a morphism of fibrations relative to $p : X^3 \to X^2$, $p(a,b,c) = (a,c)$ and we will prove the following: 

\begin{theorem}[Theorem \ref{thm : propriété de base de PPDG}]\label{theorem A PPDG}
     The product $$\textup{PP}_{DG} : H_*(X^2,C_*(\Omega Y))^{\otimes 2} \to H_*(X^2,C_*(\Omega Y))$$ of degree $-n$ defined by the composition 

$$\begin{array}{rcl}
    H_i(X^2,C_*(\Omega Y)) \otimes H_j(X^2, C_*(\Omega Y)) & \overset{(-1)^{nj}K}{\longrightarrow} &  H_{i+j}(X^4, C_*(\Omega Y^2)) \\
     &\overset{D_!}{\longrightarrow} & H_{i+j-n}(X^3,D^*C_*(\Omega Y^2)) \\
     & \overset{\Tilde{m}}{\longrightarrow} &  H_{i+j-n}(X^3, p^*C_*(\Omega Y))\\
     &  \overset{p_{*}}{\longrightarrow}  &H_{i+j-n}(X^2, C_*(\Omega Y))
\end{array}$$

is associative, admits a neutral element and corresponds, via the Fibration Theorem to the product \\ $\Lambda : H_i(\mathcal{P}_{X,f} Y) \otimes H_j(\mathcal{P}_{X,f} Y) \to H_{i+j-n}(\mathcal{P}_{X,f} Y)$ defined by Stegemeyer in \cite{Max23}.
\end{theorem}

\subsection{Homotopy invariance and module structure over the Chas-Sullivan ring}

We can reprove several results from \cite{Max23} in this setting. In particular, if $Y^k$ is a pointed, oriented, closed, and connected manifold, we prove that $H_*(X^2, C_*(\Omega Y))$ has a module structure over  $H_*(Y, C_*(\Omega Y))$ endowed with the ring structure \hyperref[CSDG]{$\CS$}.

\begin{proposition}[Proposition \ref{prop : module structure over Chas-Sullivan product}]\label{proposition : module structure over Chas-Sullivan product}
    There exists a pairing $$\cdot_M : H_i(Y, \COY) \otimes H_j(X^2, \COY) \to H_{i+j-k}(X^2, \COY)$$ that induces a left $H_*(Y,\COY)$-module structure  on $H_*(X^2, \COY)$ with the additional property that if $\gamma \in H_*(Y, \COY)$ and $\tau, \tau' \in H_*(X^2, \COY)$, then 

    $$\gamma \cdot_M \PP(\tau \otimes \tau') = \PP((\gamma \cdot_M \tau) \otimes \tau').$$

\end{proposition}

We also prove that the ring and module structures only depend on the homotopy type of $f : X \to Y$ 

\begin{proposition}[Propositions \ref{prop : ring iso homotopy type of f} and \ref{prop : module iso homotopy type of f}]\label{proposition : ring iso homotopy type of f}
     If $g : X \to Y$ is homotopic to $f$, then there exists a chain homotopy equivalence $\tilde{\Psi}_{f,g} : C_*(X, \COY ;f) \to C_*(X, \COY ;g)$ that induces an isomorphism $$\tilde{\Psi}_{f,g} : H_*(X, \COY ; f) \overset{\sim}{\to} H_*(X, \COY ; g)$$  of rings and modules.
\end{proposition}

and independent of the  homotopy equivalence class of the pair $(Y,X)$ when $X \subset Y$.

\begin{proposition}[Propositions \ref{prop : ring iso homotopy equivalence} and \ref{prop : module iso homotopy equivalence}]\label{proposition : ring and module homotopy equivalence}
    If $\varphi : (Y_1, X_1) \to (Y_2,X_2)$ is a homotopy equivalence of pairs, then there exists an isomorphism of rings $$\varphi_{\#} : H_*((X_2)^2, C_*(\Omega Y_2)) \to H_*((X_1)^2, C_*(\Omega Y_1))$$ that satisfies the equation
    $$\forall \gamma \in H_*(Y_2, C_*(\Omega Y_2)), \forall \tau \in H_*((X_2)^2, C_*(\Omega Y_2)), \ \varphi_{\#}(\gamma \cdot_M \tau) = \varphi^{\mathcal{L}}_{\#}(\gamma) \cdot_M \varphi_{\#}(\tau),$$ 
    where $\varphi^{\mathcal{L}}_{\#} = \widetilde{\mathcal{L}\varphi}^{-1} \circ \varphi_! : H_*(Y_2, C_*(\Omega Y_2)) \to H_*(Y_1, C_*(\Omega Y_1))$ is the ring isomorphism induced by the homotopy equivalence $\varphi : Y_1 \to Y_2$ (see \cite[Proposition 7.13]{Rie24}).
\end{proposition}

\textbf{Acknowledgements.} I would like to thank my PhD advisors Mihai Damian and Alexandru Oancea for their many advice and encouragements. I would also like to thank Maximilian Stegemeyer for a lot of very useful discussions.

\section{Definition}\label{section : PPDG Definition}

Let $Y$ be a topological space, $(X,\star)$ be a pointed, oriented, closed, and connected manifold  and $f : X \to Y$ be a continuous map. Denote $\star_Y = f(\star_X).$ Let
$$\mathcal{P} Y = \left\{ (\gamma,a), \ a \in [0,+\infty), \ \gamma : [0,a] \to Y \textup{ continuous} \right\}$$
be the Moore space of paths in $Y$. For $(\gamma,a) \in \mathcal{P}Y,$ we will denote $\gamma \in \mathcal{P}Y$ when the mention of the non-negative real number $a$ is not useful.  This space is naturally endowed with a fibration $$\Omega Y \hookrightarrow \mathcal{P} Y \overset{\ev_0,\ev_1}{\to} Y^2$$ given by the evaluation at the basepoint and endpoint. Consider its pullback by $f\times f : X^2 \to Y^2$,

$$\Omega Y \hookrightarrow \mathcal{P}_{X,f} Y \overset{\ev}{\to} X^2,$$

where

 $$\mathcal{P}_{X,f} Y = \{(x_0,x_1,(\alpha,a)) \in X^2 \times \mathcal{P}Y, \ \alpha(0) = f(x_0) , \alpha(a)= f(x_1)\}$$ and the evaluation is given by $$\ev(x_0,x_1,\alpha) := (\ev_0(x_0,x_1,\alpha), \ev_1(x_0,x_1,\alpha)) := (x_0,x_1).$$

We will often see an element $(x_0,x_1,\gamma) \in\mathcal{P}_{X,f} Y$ as $\gamma \in \mathcal{P} Y$ in the fiber $\ev^{-1}(x_0,x_1)$ and therefore write $\gamma \in \mathcal{P}_{X,f} Y$ when it is not necessary to precise the fiber in which $\gamma$ belong. We will consider the following \textbf{transitive lifting function} $$ \Phi_{\ev} : \mathcal{P}_{X,f} Y \ftimes{\ev}{\ev_0} \mathcal{P} X^2 \to \mathcal{P}_{X,f} Y,\  \Phi_{\ev}(\gamma, (\tau_1,\tau_2)) = (f\circ\tau_1^{-1}) \gamma (f \circ \tau_2)$$ for this fibration.
A degree $-n$ product $$\Lambda : H_*(\mathcal{P}_{X,f}Y)^{\otimes 2} \to H_*(\mathcal{P}_{X,f}Y)$$ has been defined and studied by \cite{Max23} if $Y$ is a closed manifold. This product, as the Chas-Sullivan product, is defined by intersecting on a space where the paths are concatenable and then concatenate.

$$\begin{array}{lcl}
    \Lambda : H_i(\mathcal{P}_{X,f}Y) \otimes H_j(\mathcal{P}_{X,f}Y) & \longrightarrow &  H_{i+j} (\mathcal{P}_{X,f}Y \times \mathcal{P}_{X,f}Y) \\
     & \overset{\textup{inters.}}{\longrightarrow} &  H_{i+j-n}(\mathcal{P}_{X,f}Y \ftimes{\ev_1}{\ev_0} \mathcal{P}_{X,f}Y) \\
     & \overset{\textup{concat.}}{\longrightarrow} &  H_{i+j-n}(\mathcal{P}_{X,f}Y).
\end{array}$$

The Fibration Theorem \cite[Theorem 7.2.1]{BDHO23} states that $$H_*(\mathcal{P}_{X,f} Y) \cong H_*(X^2,(f\times f)^*\COY),$$ where the $C_*(\Omega X^2)$ module structure on $(f \times f)^*\COY$ is given by the transitive lifting function $\Phi_{\ev}.$ To abbreviate the notations, we will denote this homology $H_*(X^2,\COY;f)$ or just $H_*(X^2,\COY)$ when $f$ is implied.

\begin{defi}

    Let $Y$ be a topological space, $X$ be a $n$-dimensional closed, oriented, path-connected and pointed manifold and $f : X \to Y$ be a continuous map. We define a degree $-n$ product $$\PP : H_*(X^2, \COY)^{\otimes 2} \to H_*(X^2, \COY)$$ by the composition of the four maps

$$\begin{array}{crl}
     & H_i(X^2, \COY) \otimes H_j(X^2, \COY)  &\\
    \overset{(-1)^{nj}K}{\longrightarrow} & H_{i+j}(X^4, C_*(\Omega Y^2)) & \simeq H_{i+j}(\mathcal{P}_{X,f} Y \times \mathcal{P}_{X,f} Y)   \\
     \overset{D_!}{\longrightarrow} & H_{i+j-n}(X^3, D^* C_*(\Omega Y^2)) & \simeq H_{i+j-n}(\mathcal{P}_{X,f} Y \ftimes{\ev_1}{\ev_0} \mathcal{P}_{X,f} Y) \\
     \overset{\tilde{m}}{\longrightarrow} & H_{i+j-n}(X^3, p^*\COY) & \simeq H_{i+j-n}(m(\mathcal{P}_{X,f} Y \ftimes{\ev_1}{\ev_0} \mathcal{P}_{X,f} Y)) \\
     \overset{p_*}{\longrightarrow} & H_{i+j-n}(X^2, \COY) & \simeq H_{i+j-n}(\mathcal{P}_{X,f} Y),
\end{array}$$

where :

\begin{itemize}
    \item $K : H_i(X^2, \COY) \otimes H_j(X^2, \COY) \to H_{i+j}(X^4, C_*(\Omega Y^2))$ is the topological DG Künneth map defined in \cite[Theorem 6.14]{Rie24},
    
    \item $D_! : H_{i+j}(X^4, C_*(\Omega Y^2)) \to H_{i+j-n}(X^3, D^* C_*(\Omega Y^2))$ is the shriek map of $D : X^3 \to X^4, \ D(a,b,c) = (a,b,b,c)$ and $p_* : H_{i+j-n}(X^3, p^*\COY) \to H_{i+j-n}(X^2, \COY)$ is the direct map of $p : X^3 \to X^2, \ p(a,b,c) = (a,c)$ as defined in \cite[Sections 9 and 10]{BDHO23},
    
    \item $\tilde{m} :  H_{i+j-n}(X^3, \Delta^* C_*(\Omega Y^2)) \to H_{i+j-n}(X^3, p^*\COY) $ is the morphism associated with the morphism of fibrations given by the concatenation $m : \mathcal{P}_{X,f} Y \ftimes{ev_1}{ev_0} \mathcal{P}_{X,f}Y \to p^*\mathcal{P}_{X,f}Y,$ (see \cite[Theorem 5.8]{Rie24}).
\end{itemize}

\end{defi}

\begin{rem}
    Intuitively, given two chains $\sigma, \tau$ in $\mathcal{P}_{X,f}Y$, this product will intersect $\ev_{1,*}(\sigma)$ with $\ev_{0,*}(\tau)$ on $X$, concatenate on the intersection and forget the intersection points in which it concatenated.

    We can remark that since the intersection is not along the diagonal $\Delta: X^2 \to X^4$, this product does not fall under the category of products studied in \cite{Rie24}. Nonetheless, the tools developed there enable to define and study it. 
\end{rem}

The reader may wonder why it takes two maps $\tilde{m} : H_*(X^3, p^*\COY)$ and $p_* : H_*(X^2, \COY)$ to translate in Morse theoretic terms the concatenation $m : H_*(\mathcal{P}_{X,f} Y \ftimes{\ev_1}{\ev_0} \mathcal{P}_{X,f} Y) \to H_*(\mathcal{P}_{X,f}).$ Here the fiber product $\mathcal{P}_{X,f} Y \ftimes{\ev_1}{\ev_0} \mathcal{P}_{X,f} Y$ is seen as the total space of the fibration 

$$\begin{array}{ccc}
    \mathcal{P}_{X,f} Y \ftimes{\ev_1}{\ev_0} \mathcal{P}_{X,f} Y & \to & X^3 \\
     (\gamma_1,\gamma_2) & \mapsto & (\ev_0(\gamma), \ev_1(\gamma_1), \ev_1(\gamma_2)) ,
\end{array}$$

while $\mathcal{P}_{X,f}Y$ is a fibration over $X^2$. Since the bases are not the same, the concatenation $m$ is not a morphism of fibrations in the sense given in \cite[Theorem 5.8]{Rie24}. 

Using \cite[Proposition 9.8.1]{BDHO23}, we obtain a direct generalization of \cite[Theorem 5.8]{Rie24} for relative morphism of fibrations (see Definition \ref{defi : relative morphism of fibrations}).

\begin{thm}\label{cor : morphism of fibrations relative to ...}
    Let $g: Z_0 \to Z_1$ be continuous map between oriented, closed, and connected manifolds each endowed with a fibration $F_0 \hookrightarrow E_0 \overset{\pi_0}{\to} Z_0$ and $F_1 \hookrightarrow E_1 \overset{\pi_1}{\to} Z_1.$ Let $\varphi : E_0 \to E_1$ be a morphism of fibrations relative to $g.$ Then $\varphi_* : C_*(E_0) \to C_*(E_1)$, corresponds, via the Fibration Theorem, up to chain homotopy equivalence, to the composition 
    $$g_*\circ \tilde{\varphi}_0 : C_*(Z_0,C_*(F_0)) \to C_*(Z_0, g^*C_*(F_1)) \to C_*(Z_1, C_*(F_1)).$$
    In other words,
    the following diagram is commutative up to chain homotopy:

    $$\xymatrix{
    C_*(Z_0, C_*(F_0)) \ar[d]^{\Psi_{E_0}} \ar[r]^{\tilde{\varphi}_0} & C_*(Z_0, g^*C_*(F_0)) \ar[r]^{g_*} & C_*(Z_1, C_*(F_1)) \ar[d]^{\Psi_{E_1}} \\
    C_*(E_0) \ar[rr]^{\varphi_*} &  & C_*(E_1).
    }$$
\end{thm}

\begin{proof}
     The map $\varphi$ factors through $$\varphi : \begin{array}{ccccc}
         E_0 & \overset{\varphi_0}{\to} & g^*E_1 & \to & E_1  \\
         e & \mapsto & (\pi_0(e), \varphi(e)) & \mapsto & \varphi(e) ,
    \end{array}$$ 

where $\varphi_0 :  E_0 \to g^* E_1$ is a morphism of fibrations over $Z_0.$

    $$\xymatrix{
E_0 \ar[d]_{\pi_0} \ar[r]  \ar@/^1pc/[rr]^{\varphi} & g^*E_1 \ar[r] \ar[d]^{pr_1} & E_1 \ar[d]^{\pi_1} \\
Z_0 \ar@{=}[r] & Z_0 \ar[r]_g & Z_1.
}$$
    
    Therefore, using \cite[Theorem 5.8]{Rie24} and \cite[Proposition 9.8.1]{BDHO23} we obtain that the following diagram is commutative up to chain homotopy:

    $$\xymatrix{
    C_*(Z_0, C_*(F_0)) \ar[d]^{\Psi_{E_0}} \ar[r]^-{\tilde{\varphi}_0} & C_*(Z_0, h^*C_*(F_1)) \ar[r]^{h_*} & C_*(Z_1, C_*(F_1)) \ar[d]^{\Psi_{E_1}} \\
    C_*(E_0) \ar[rr]^{\varphi_*} &  & C_*(E_1).
    }$$
\end{proof}

Using very similar techniques as those used in \cite[Section 7]{Rie24} to prove \cite[Theorem 7.1]{Rie24}, we will prove:

\begin{thm}[Theorem \ref{theorem A PPDG}]\label{thm : propriété de base de PPDG}
    The product $\PP$ is associative admits a neutral element and corresponds to the product $$\Lambda : H_{i}(\mathcal{P}_{X,f}Y) \otimes H_j(\mathcal{P}_{X,f} Y) \to H_{i+j-n}(\mathcal{P}_{X,f}Y)$$ defined by Stegemeyer \cite{Max23}. 
\end{thm}

\section{Proof of Theorem \ref{theorem A PPDG}}\label{section : PPDG Proof of Theorem}

This product uses both direct maps and shriek maps. Therefore, we need to understand how these maps interact with each other. The only pre-existing result exploring this question is the following:

\begin{prop}\textup{(\cite[Proposition 10.6.1]{BDHO23})}
    Let $\varphi : X \to Y$ be a map of degree $d \in \Z$ between closed oriented manifolds of equal dimensions. Let $\F$ be a DG local system on $Y$. Then 
    $$\varphi_* \varphi_! = d\cdot \Id : H_*(Y,\F) \to H_*(Y,\F).$$
    In particular, if $d = \pm 1$, the map $\varphi_!$ is injective and the map $\varphi_*$ is surjective.
\end{prop}

We will use this result to prove the following lemma:

\begin{lemme}\label{lemme : commutativity shriek and direct}
    Let $i_X : X^n \hookrightarrow W^{n+k}$ and $i_Y : Y^m \hookrightarrow Z^{m+k}$ be two $k-$codimensional embeddings of oriented, closed, and connected manifolds. Let $\F$ be a DG local system on $Z$ and $\Xi_X$, $\Xi_Y, \Xi_Z$ and $\Xi_W$ be sets of DG Morse data. Let $\varphi : W \to Z$ and $\psi : X \to Y$ such that $$\varphi \circ i_X = i_Y \circ \psi : X \to Z \textup{ and } \varphi^{-1}(i_Y(Y)) = i_X(X).$$ Assume that the normal bundle $N_X$ of $i_X$ and the normal bundle $N_Y$ of $i_Y$ are isomorphic and denote $\varepsilon \in \{-1,1\}$ such that $$\Or \ N_X = \varepsilon \cdot \Or \ N_Y.$$
    
    Then, the maps $$\varepsilon \cdot (\psi_* \circ i_{X,!} )\textup{ and } i_{Y,!} \circ \varphi_* : C_*(W, \Xi_W, \varphi^*\F) \to C_{*-k}(Y, \Xi_Y, i_Y^*\F), $$
    are chain homotopic.
\end{lemme}

\begin{proof}
   \textbf{\underline{Step 1:}} First assume that $X \subset W$ and $Y \subset Z$ are submanifolds. Then, $\psi = \varphi_{\lvert_X} : X \to Y$ and $\varphi^{-1}(Y) = X.$ It suffices to prove the result for well-chosen data $\Xi_X$, $\Xi_Y, \Xi_Z$ and $\Xi_W$, since the continuation maps are chain homotopy equivalence. Recall that since $X,Y,Z$ and $W$ are oriented manifolds, the orientations of the normal bundles are given by 
   
   $$\left(\Or \ Y, \Or \ N_Y\right) = \Or \ Z$$ and
    $$\left(\Or \ X, \Or \ N_X\right) = \Or \ W.$$
    Choose $\Xi_X$ and let $\Xi_W$ be the set of DG Morse data described in \cite[Section 9.2]{BDHO23} given by considering a tubular neighborhood $U \overset{i_U}{\subset} W$ of $X$ and extending the Morse-Smale pair $(f,\xi) \in \Xi_X$ in $(f_U,\xi_U) \in \Xi_U$ by taking the pseudo-gradient $\xi_U$ to point outwards $U$. 
    Orient the unstable manifolds by 
    $$
    \Or \ \wbi{u}{U}{x} = \left( \Or \ \wbi{u}{X}{x}, \Or \ N_X \right).
    $$
    With such data, $$C_*(X,\Xi_X,i_X^*\varphi^*\F) = C_{*+k}(U, \Xi_U, i_U^*\varphi^*\F).$$
    Then, extend $(f_U,\xi_U)$ to $(f_W,\xi_W) \in \Xi_W$. The shriek map of the inclusion is then given by
    $$\deffct{i_{X,!}}{C_*(W,\Xi_W,\varphi^*\F)}{C_{*-k}(X,\Xi_X, i_X^*\varphi^*\F)}{\alpha \otimes w}{\left\{ \begin{array}{cc}
        \alpha \otimes w & \textup{if } w\in X  \\
         0 & \textup{otherwise.} 
    \end{array} \right.}$$
    
    Apply the same procedure from a set of DG Morse data $\Xi_Y$ to obtain $\Xi_Z$ such that 

    $$\deffct{i_{Y,!}}{C_*(Z,\Xi_Z,\F)}{C_{*-k}(Y,\Xi_Y, i_Y^*\F)}{\alpha \otimes z}{\left\{ \begin{array}{cc}
        \alpha \otimes z & \textup{if } z\in Y  \\
         0 & \textup{otherwise.} 
    \end{array} \right.}$$

Denote $f_Z : Z \to \R$ and $f_Y : Y \to \R$ the Morse functions in the data sets $\Xi_Z$ and $\Xi_Y$ respectively.
With these data, the composition $\varphi_{\lvert_{X,*}} \circ i_{X,!}$ is given by $$\deffct{\varphi_{\lvert_{X,*}} \circ i_{X,!}}{C_*(W,\Xi_W,\varphi^*\F)}{C_{*-k}(Y,\Xi_Y, i_Y^*\F)}{\alpha \otimes w}{\left\{ \begin{array}{cc}
    \displaystyle \sum_{y} \alpha \cdot \nu^{\varphi_{\lvert_X}}_{w,y} \otimes y & \textup{if } w \in X  \\
     0 & \textup{otherwise,} 
\end{array} \right.}$$
    where $\nu^{\varphi_{\lvert_X}}_{w,y} \in C_{|w|-|y|}(\Omega Y)$ is the evaluation onto $\Omega Y$ of a chain representative system of the space of trajectories $$\functraj{\varphi_{\lvert_X}}{w,y} = \wbi{u}{f_X}{w} \cap \varphi_{\lvert_{X}}^{-1}\left( \wbi{s}{f_Y}{y} \right).$$

    The composition $i_{Y,!} \circ \varphi_* $ is given by $$\deffct{i_{Y,!} \circ \varphi_*}{C_*(W,\Xi_W,\varphi^*\F)}{C_{*-k}(Y,\Xi_Y, i_Y^*\F)}{\alpha \otimes w}{\left\{ \begin{array}{cc}
    \displaystyle \sum_{y \in \Crit(f_Y)} \alpha \cdot \nu^{\varphi}_{w,y} \otimes y & \textup{if } w \in X  \\
     0 & \textup{otherwise.} 
\end{array} \right.}$$

Indeed, for all $w \in \Crit(f_W)$,  $$\varphi_*(\alpha \otimes w) = \sum_{z \in \Crit(f_Z)} \alpha \cdot \nu^{\varphi}_{w,z} \otimes z,$$ and there are three cases:

\begin{enumerate}
    \item If $z \notin Y$, then $i_{Y,!}(\alpha \cdot \nu^{\varphi}_{w,z} \otimes z) = 0.$\\

    \item If $w \notin X$ and $z \in Y$, then $\wbi{u}{f_W}{w} \cap X = \emptyset $ since the pseudo-gradient $\xi_W$ points outwards $U$ and therefore, a trajectory that begins outside of $U$ cannot reach $U$. Moreover, for the same reason, $\wbi{s}{f_Z}{z} = \wbi{s}{f_Y}{z} \subset Y$ and since we assume that $\varphi^{-1}(Y) = X$, it follows that $\varphi^{-1}\left(\wbi{s}{f_Z}{z} \right) \subset X$.  In particular  $\functraj{\varphi}{w,z} = \wbi{u}{f_W}{w} \cap \varphi^{-1}\left( \wbi{s}{f_Z}{z} \right) = \emptyset$ and $\nu^{\varphi}_{w,z} = 0.$\\

    \item If $w \in X$ and $z \in Y$, then $i_{Y,!}(\alpha \cdot \nu^{\varphi}_{w,z} \otimes z) = \alpha \cdot \nu^{\varphi}_{w,z} \otimes z.$
\end{enumerate}

    To conclude this first step, it suffices to prove that $$\forall w \in \Crit_i(f_X)\subset \Crit_{i+k}(f_W) , \forall z \in \Crit_j(f_Y) \subset \Crit_{j+k}(f_Z), \quad \functraj{\varphi_{\lvert_X}}{w,z} = \functraj{\varphi}{w,z}$$ and compare their orientations.

    Let $a \in \functraj{\varphi_{\lvert_X}}{w,z} = \wbi{u}{f_X}{w} \cap \varphi_{\lvert_X}^{-1}\left(\wbi{s}{f_Y}{z}\right).$ Then, $a \in \wbi{u}{f_W}{w} \cap \varphi^{-1}\left(\wbi{s}{f_Z}{z}\right) = \functraj{\varphi}{w,z} $ since $\wbi{u}{f_X}{w} \subset \wbi{u}{f_W}{w}$ and $\wbi{s}{f_Y}{z} = \wbi{s}{f_Z}{z}$.
    Therefore $$\functraj{\varphi_{\lvert_X}}{w,z} \subset \functraj{\varphi}{w,z}.$$

    Let $a \in \functraj{\varphi}{w,z}.$ In particular $a \in \varphi^{-1}\left(\wbi{s}{f_Z}{z}\right) = \varphi^{-1}\left(\wbi{s}{f_Y}{z}\right) \subset \varphi^{-1}(Y) = X.$ Therefore, $$a \in \left( \wbi{u}{f_W}{w} \cap X \right) \cap \varphi_{\lvert_X}^{-1}\left(\wbi{s}{f_Y}{z}\right) = \wbi{u}{f_X}{w} \cap \varphi_{\lvert_X}^{-1}\left(\wbi{s}{f_Y}{z}\right) $$

    and $$\functraj{\varphi}{w,z} \subset \functraj{\varphi_{\lvert_X}}{w,z}.$$

    It follows that $\functraj{\varphi}{w,z} = \functraj{\varphi_{\lvert_X}}{w,z}$ and we now compare their orientations.

    $$\left( \Or \ \functraj{\varphi}{w,z}, \Or \ \wbi{u}{Z}{z}\right) = \Or \ \wbi{u}{W}{w}.$$

    \begin{align*}
        \left( \Or \ \functraj{\varphi_{\lvert_X}}{w,z}, \Or \ \wbi{u}{Z}{z} \right) &= \left( \Or \ \functraj{\varphi_{\lvert_X}}{w,z}, \Or \ \wbi{u}{Y}{z}, \Or \ N_Y\right) \\
        &= \left( \Or \ \wbi{u}{X}{w}, \Or \ N_Y \right)\\
        &= \varepsilon \left( \Or \ \wbi{u}{X}{w}, \Or \ N_X \right) \\
        &= \varepsilon \cdot  \Or \ \wbi{u}{W}{w}.
    \end{align*} 

    Therefore, with these data, for all $w\in \Crit(f_X) \subset \Crit(f_W)$ and $y \in \Crit(f_Y) \subset \Crit(f_Z),$ $$\nu^{\varphi_{\lvert_X}}_{w,y} = \varepsilon \cdot \nu^{\varphi}_{w,z}$$ and $$\varepsilon\cdot (\varphi _{\lvert_{X,*}} \circ i_{X,!})= i_Y \circ \psi.$$

     \textbf{\underline{Step 2:}} In the general case, we consider $i_X : X \to W$ as the composition of the orientation-preserving diffeomorphism $i_X : X \to i_X(X)$  with the inclusion $i: i_X(X) \to W$. Denote $j : i_Y(Y) \to Z$ the inclusion.

    Consider the following diagram:
     $$\xymatrix{
     C_{*-k}(X,\Xi_X,i_X^*\varphi^*\F)  \ar[d]^{\psi_*} & C_{*-k}(i_X(X), i_X^*\Xi_X, i^*\varphi^*\F) \ar[l]^{i_{X,*}^{-1}= i_{X,!}} \ar[d]^{\varphi_{\lvert_{i_X(X),*}}} & C_{*}(W,\Xi_W, \varphi^*\F) \ar[l]^-{i_!} \ar[d]^{\varphi_*} \\
     C_{*-k}(Y,\Xi_Y, i_Y^*\F) & C_{*-k}(i_Y(Y), i_Y^*\Xi_Y, j^*\F) \ar[l]^-{i_{Y,*}^{-1} = i_{Y,!}} & C_*(Z,\Xi_Z, \F) \ar[l]^-{j_!}.
     }$$

     We proved in the first step that the right square commutes up to chain homotopy up to the sign $\varepsilon$. The fact that $i_{X,!} = i_{X,*}^{-1}$ and $i_{X,!} = i_{X,*}^{-1}$ is a direct consequence of \cite[Proposition 10.6.1]{BDHO23} stated above. The left square commutes up to chain homotopy by the composition property of direct maps. 
\end{proof}

We now prove that the product $\PP$ is associative and admits a neutral element.

\begin{prop}
    The product $\PP$ is associative at the homology level.
\end{prop}

\begin{proof}

We denote $\F^k = C_*(\Omega Y^k)$, $\F = \F^1$ and $H_* = H_*(X^2,\F)$. Consider the diagram

\[
\footnotesize
\xymatrix
@C=20pt
{
H_*^{\otimes 3} \ar[r]^-{K \otimes 1} \ar[d]^{1 \otimes K} & H_*(X^4,\F^2) \otimes H_* \ar[d]^K \ar[r]^-{D_! \otimes 1} \ar@{}[rd]|{\blue{(1)}} & H_*(X^3,D^*\F^2) \otimes H_* \ar[d]^K \ar[r]^-{\Tilde{m} \otimes 1} &  H_*(X^3,p^*\F^2) \otimes H_* \ar[d]^K \ar[r]^-{p_{*} \otimes 1} &  H_*^{\otimes 2} \ar[d]^{K}\\
H_* \otimes H_*(X^4,\F^2) \ar[d]^{1 \otimes D_!} \ar[r]^-{K} \ar@{}[rd]|{\blue{(2)}} & H_*(X^6,\F^3) \ar[d]^{(1 \times D)_!} \ar[r]^-{(D \times 1)_!} & H_*(X^5, D^*\F^2 \times \F) \ar[r]^-{\widetilde{m\times 1}} \ar[d]^{\Delta_{R,!}} & H_*(X^5, p^*\F^2 \times \F) \ar[r]^-{(p\times 1)_*} \ar[d]^{\Delta_{R,!}} \ar@{}[rd]|{\blue{(3)}} & H_*(X^4,\F^2) \ar[d]^{D_!} \\
H_* \otimes H_*(X^3,D^*\F^2) \ar[d]^{1 \otimes \Tilde{m}} \ar[r]^-{K} & H_*(X^5, \F\times D^*\F^2) \ar[r]^-{\Delta_{L,!}} \ar[d]^{\widetilde{1\times m}} & H_*(X^4, \Delta_2^*\F^3) \ar[d]^{\widetilde{1\times m}} \ar[r]^-{\widetilde{m\times 1}} & H_*(X^4,\Delta p_{\hat{2}}^*\F^2) \ar[r]^-{p_{\hat{2},*}} \ar[d]^{\Tilde{m}} & H_*(X^3,\Delta^*\F^3)\ar[d]^{\Tilde{m}} \\
H_* \otimes H_*(X^3,p^*\F^2) \ar[d]^{1 \otimes p_{*}} \ar[r]^-{K} & H_*(X^5, \F\times p^*\F^2) \ar[r]^-{\Delta_{L,!}} \ar[d]^{(1\times p)_*} \ar@{}[rd]|{\blue{(4)}} & H_*(X^4, p_{\hat{3}}^*D^*\F^2) \ar[d]^{p_{\hat{3},*}} \ar[r]^-{\Tilde{m}} & H_*(X^4, p_{1,4}^*\F^2) \ar[r]^-{p_{\hat{2},*}} \ar[d]^{p_{\hat{3},*}} & H_*(X^3,p^*\F) \ar[d]^{p_*} \\
H_*^{\otimes 2} \ar[r]^-{K} & H_*(X^4,\F^2) \ar[r]^-{D_!} & H_{*}(X^3,D^*\F^2) \ar[r]^-{\Tilde{m}} & H_{*}(X^3, p^*\COY)\ar[r]^-{p_{*}} & H_{*}
}
\]

where the embeddings $\Delta_R, \Delta_L : X^4 \to X^5$ and $\Delta_2 := (D \times 1) \circ \Delta_R : X^4 \to X^6$ are defined by 

$$\Delta_L = \Id \times \Delta \times \Id_{X^2}, \ \Delta_R = \Id_{X^2} \times \Delta \times \Id \textup{ and } \Delta_2 = \Id \times \Delta \times \Delta \times \Id$$

and the projections $p_{\hat{3}}, p_{\hat{2}} : X^4 \to X^3$ and $p_2 : X^4 \to X^2$ are defined by 

$$p_{\hat{3}}(a,b,c,d) = (a,b,d), \ p_{\hat{2}}(a,b,c,d) =(a,c,d) \textup{ and } p_{1,4} := p \circ p_{\hat{2}} = p \circ p_{\hat{3}} .$$

\underline{\textbf{The square \blue{(4)} commutes:}} This is a consequence of Lemma \ref{lemme : commutativity shriek and direct}. Indeed, since \\$D\circ p_{\hat{2}} = (p \times 1) \circ \Delta_R : X^4 \to X^4$, $(a,b,c,d) \mapsto (a,c,c,d)$ and we have the equality
$$(p\times 1)^{-1}(D(X^3)) = \{(a,b,c,c,d) \in X^5\} = \Delta_R(X^4).$$ 
We now compare the orientations of the normal bundle $N_{\Delta_R}$ of $\Delta_R(X^4) \subset X^5$ with the normal bundle $N_{D}$ of $D(X^3) \subset X^4.$
Choose the orientation $\Or \ N_{\Delta} = \Or \ X$ of the normal bundle $N_{\Delta}$ of $\Delta(X) \subset X^2$. Following \cite[Remark 1.1]{hingston2022product}, this induces $$\left(\Or \ T\Delta(X), \Or \ N_{\Delta} \right) = \Or \ T(X \times X) \lvert_{\Delta(X)}.$$

Therefore
\begin{align*}
    \left( \Or \ T\Delta_{R,*}(X^4),  \Or \ N_{\Delta} \right) &= \left( \Or \ TX, \Or \ TX, \Or \ T\Delta(X), \Or \ TX,  \Or \ N_{\Delta} \right)\\
    &= (-1)^n\left( \Or \ TX, \Or \ TX, \Or \ T(X \times X) \lvert_{\Delta(X)}, \Or \ TX \right) \\
    &= (-1)^n \Or \ T(X^5) \lvert_{\Delta_R(X^4)}.
\end{align*}

and it follows that $\Or \ N_{\Delta_R} = (-1)^n \Or \ N_{\Delta} = (-1)^n \Or \ X.$ We prove using the same arguments that $\Or \ N_D = (-1)^n \Or \ X = \Or \ N_{\Delta_R} $
and that the mentioned square commutes.

\underline{\textbf{The square \blue{(3)} commutes up to the sign} $(-1)^n$:} We use the same techniques as for square \blue{(4)}. We first check that $(1 \times p) \circ \Delta_L = D \circ p_{\hat{3}} : X^4 \to X^4$, $(a,b,c,d) \mapsto (a,b,b,d)$ and
$$(1 \times p)^{-1}(D(X^3)) = \{(a,b,b,c,d) \in X^5\} = \Delta_L(X^4).$$
The normal bundles $N_{\Delta_L} \cong TX \cong N_D$ are isomorphic and their orientations differ by $(-1)^n$.

\underline{\textbf{The squares \blue{(1)} and \blue{(2)} commute up to the sign $(-1)^{nk}$ and  $(-1)^{ni}$ respectively:}} This is a direct consequence of \cite[Lemma 6.17]{Rie24}. 

All the other squares are commutative by compatibility properties. Therefore,  \begin{multline*}
    p_{*}\circ \tilde{m} \circ \Delta_! \circ K \circ ( p_{*} \otimes 1) \circ (\tilde{m} \otimes 1) \circ (\Delta_! \otimes 1) \circ ( K \otimes 1) \\
     =  (-1)^{n(n-k)} (-1)^{ni} p_{*}\circ \tilde{m} \circ \Delta_! \circ K \circ (1 \otimes p_{*}) \circ (1 \otimes \tilde{m}) \circ (1 \otimes \Delta_!) \circ (1 \otimes K).
\end{multline*}

Let $\gamma \in H_i(X^2,C_*(\Omega Y))$, $\tau \in H_j(X^2,C_*(\Omega Y))$ and $\delta \in H_k(X^2,C_*(\Omega Y))$. 
\begin{align*}
    &\PP(\gamma \otimes \PP(\tau \otimes \delta))\\
    & =(-1)^{n(n-k)} (-1)^{n(n-j-k+n)} (-1)^{ni} p_{*}\circ \tilde{m} \circ \Delta_! \circ K \circ (1 \otimes p_{*}) \circ (1 \otimes \tilde{m}) \circ (1 \otimes \Delta_!) \circ (1 \otimes K)\\
    &(\gamma \otimes \tau \otimes \delta)
\end{align*}
and 
\begin{align*}
    &\PP(\PP(\gamma \otimes \tau) \otimes \delta)\\
    & =(-1)^{n(n-k)} (-1)^{n(n-j)}  p_{*}\circ \tilde{m} \circ \Delta_! \circ K \circ ( p_{*} \otimes 1) \circ (\tilde{m} \otimes 1) \circ (\Delta_! \otimes 1) \circ ( K \otimes 1)\\
    &(\gamma \otimes \tau \otimes \delta).
\end{align*}

 We conclude $$\PP( \gamma \otimes \PP(\tau \otimes \delta)) = \PP(\PP(\gamma \otimes \tau) \otimes \delta).$$

\end{proof}

Consider the inclusion of constant paths $s: X \to \mathcal{P}_{X,f} Y$, $s(a) = (a,a,c_{f(a)})$, where we denoted $c_{f(a)}$ the constant path at $f(a) \in Y.$ This is a morphism of fibrations relative to $\Delta : X \to X^2$.

\begin{prop}
     The product $\PP$ admits the neutral element $$\Delta_*\tilde{s}([X]) \in H_{n}(X^2,\COY)$$ denoted $1_{PP}$.
\end{prop}

\begin{proof}
    Let $J = \Id \times \Delta : X^2 \to X^3$, $(a,b) \mapsto (a,b,b).$ Denote $\F^i = C_*(\Omega Y^i)$ and $\sigma : \mathcal{P}_{X,f} \ftimes{\ev_1}{\Id} X \to \mathcal{P}_{X,f}$, the morphism of fibrations $\sigma(\gamma, \ev_1(\gamma)) = \gamma$ over $X^2.$
   Consider the following diagram:

   $$\small \xymatrix{
   H_*(X^2,\F) \otimes H_n(X,\Z) \ar[d]^{K} \ar[r]^-{1 \otimes \tilde{s}} & H_*(X^2,\F) \otimes H_n(X, \Delta^*\F) \ar[d]^K \ar[r]^-{1 \otimes \Delta_*} & H_*(X^2, \F) \otimes H_n(X^2, \F) \ar[d]^K  \\
   H_{*+n}(X^3, C_*(\Omega Y \times \{\star\})) \ar[d]^{J_!} \ar[r]^{\widetilde{1 \times s}} &  H_{*+n}(X^3, (1 \times \Delta)^*\F^2) \ar[d]^{J_!} \ar[r]^-{(1 \times \Delta)_*} & H_{*+n}(X^4, \F^{2}) \ar[d]^{D_!}   \\
   H_*(X^2, J^* C_*(\Omega Y \times \{\star\})) \ar[r]^{J^*\widetilde{1 \times s}} \ar[dr]^{\tilde{\sigma}} 
    & H_*(X^2, J^*(1 \times \Delta)^*\F^{2}) \ar[d]^{J^*\tilde{m}} \ar[r]^-{J_*}& H_*(X^3, D^* \F^{2}) \ar[d]^{\tilde{m}} \\
     & H_*(X^2,\F) \ar[r]^{J_*} \ar@{=}[dr] & H_*(X^3, p^*\F) \ar[d]^{p_*} \\
     & & H_*(X^2,\F).
   }$$

Remark that $p_* \circ J_* = (p \circ J)_* = \Id_* : H_*(X^2,\F) \to H_*(X^2,\F)$.
The map
$$m\circ (1 \times s) : \mathcal{P}_{X,f} Y \ftimes{\ev_1}{\Id} X \to \mathcal{P}_{X,f} Y \ftimes{\ev_1}{\ev} \Delta^*\mathcal{P}_{X,f} Y \to \mathcal{P}_{X,f} Y $$
$$m \circ (1 \times s) (\gamma, \ev_1(\gamma)) = m(\gamma,c_{\ev_1(\gamma)}) = \gamma$$

is equal to $\sigma,$ where we denoted $c_{\ev_1(\gamma)}$ the constant path at $\ev_1(\gamma).$ Therefore, $$J^*\tilde{m}\circ J^*\widetilde{1 \times s} = \tilde{\sigma}.$$

Since, $(1 \times \Delta) \circ J = D \circ J : X^2 \to X^4$, $(a,b) \mapsto (a,b,b,b)$ and
$$(1 \times \Delta)^{-1}(D(X^3)) = \{(a,b,b), \ (a,b) \in X^2\} = J(X^2),$$ it follows from Lemma \ref{lemme : commutativity shriek and direct} that $$D_! \circ (1 \times \Delta)_* = (-1)^n J_* \circ J_!.$$
Indeed, using the same arguments as in the previous proof, we prove that $\Or \ N_J = \Or \ X$ and $\Or \ N_D = (-1)^{n} \Or \ X.$
Therefore, this diagram commutes up to the sign $(-1)^n$.

Let $h : X \to \R$ be a Morse function with a unique local maximum $x_{\max} \in \Crit_n(h)$ such that $$x_{\max} = \star \otimes x_{\max} = [X] \in H_n(X,C_*(\{\star\})) = H_n(X,\Z).$$ Let $\xi$ be a pseudo-gradient associated with $h$. It suffices to check that $$\tilde{\sigma} \circ J_! \circ K(\cdot \otimes( \star \otimes x_{\max})) = \Id : H_*(X^2,\F) \to H_*(X^2,\F).$$

Let $v_0,v_1 : X \to \R$ be Morse functions and $\xi_0,\xi_1$ be associated pseudo-gradients. Denote $v=v_1 + v_2$ and $\xi_2 = \xi_0 + \xi_1.$ Consider $\Xi_2 = (v,\xi_2, \dots)$ be a set of DG Morse data on $X^2$ and $\Xi_3 = (v + h, \xi_2 + \xi, \dots)$ be a set of Morse data on $X^3.$

Let $\alpha \otimes (x,y) \in C_*(X^2,\Xi_2,C_*(\Omega Y))$ and consider $$J_! : C_*(X^3, \Xi_3, C_*(\Omega Y \times \{\star \})) \to C_{*-n}(X^2, \Xi_2, J^* C_*(\Omega Y \times \{\star \})).$$
Notice that
$$\overline{\mathcal{M}}^{J_!}((x,y,x_{\max}),(z,w)) = \wb{s}{z,w} \cap J^{-1}\left(\wb{u}{x,y,x_{\max}}\right) = \wb{s}{z,w} \cap \wb{u}{x,y} = \overline{\mathcal{M}}^{\Id_!}((x,y),(z,w)).$$

Since $\Or \ \wb{s}{x,y,x_{\max}} = \Or \ \wb{s}{x,y}$, these manifolds have the same orientation $$\left( \Or \ \wb{s}{x,y}, \Or \ \overline{\mathcal{M}}^{\Id_!}((x,y),(z,w)) \right) = \Or \ \wb{s}{z,w}$$ and $$\left( \Or \ \wb{s}{x,y}, \Or \ \overline{\mathcal{M}}^{J_!}((x,y,x_{\max}),(z,w)) \right) = \Or \ \wb{s}{z,w}.$$
 Hence, $J_!((\alpha, \star) \otimes (x,y,x_{\max})) = \Id_!((\alpha, \star) \otimes (x,y)).$

Since $\sigma_{*} : J^*C_*(\Omega Y \times \{\star\}) \to C_*(\Omega Y)$ is a morphism of $C_*(\Omega X^2)$-modules, $\tilde{\sigma}((\alpha, \star) \otimes (x,y)) = \alpha \otimes (x,y),$ 
then, $$\tilde{\sigma} \circ J_! \circ K (\cdot \otimes (\star \otimes x_{\max})) = \Id : H_*(X^2,C_*(\Omega Y)) \to H_*(X^2, C_*(\Omega Y)).$$

It follows that $$\PP(\tau \otimes \Delta_*\tilde{s}(\star \otimes x_{max})) = (-1)^{n^2} (-1)^n \tilde{\sigma}\circ J_! \circ K(\tau \otimes (\star \otimes x_{max})) = \tau$$ for all $\tau \in H_*(X^2, C_*(\Omega Y))$.
With similar arguments, we prove that $$\PP(\Delta_*\tilde{s}(\star \otimes x_{max}) \otimes \tau) = \tau$$ for all $\tau \in H_*(X^2, C_*(\Omega Y))$.

\end{proof}

We conclude by proving that this product computes the product defined by Stegemeyer in \cite{Max23}.

\begin{prop}
    The products $$\Lambda : H_{i}(\mathcal{P}_{X,f}Y) \otimes H_j(\mathcal{P}_{X,f} Y) \to H_{i+j-n}(\mathcal{P}_{X,f}Y)$$ and $$\PP : H_i(X^2,C_*(\Omega Y)) \otimes H_j(X^2,C_*(\Omega Y)) \to H_{i+j-n}(X^2, C_*(\Omega Y))$$ correspond via the Fibration Theorem.
\end{prop}

\begin{proof}

    This statement amounts to the commutativity of the following diagram.

    $$
    \xymatrix{
    H_i(\mathcal{P}_{X,f}Y) \otimes H_j(\mathcal{P}_{X,f}Y) \ar[d]^-{(-1)^{n-ni}\times} &  H_i(X^2,C_*(\Omega Y)) \otimes H_j(X^2,C_*(\Omega Y)) \ar[d]^-{(-1)^{nj}K} \ar[l]^-{\Psi \otimes \Psi}_-{\sim} \\
    H_{i+j}(\mathcal{P}_{X,f}Y \times \mathcal{P}_{X,f}Y) \ar[d]^-{R_{C^f,*} \circ \tau_{C^f} \cap} & H_{i+j}(X^4,C_*(\Omega Y^2)) \ar[d]^-{D_!} \ar[l]_-{\sim} \\
     H_{i+j-n}(\mathcal{P}_{X,f}Y \ftimes{\ev_1}{\ev_0} \mathcal{P}_{X,f}Y)  \ar[dd]^-{m_*} & H_{i+j-n}(X^3, D^*C_*(\Omega Y^2)) \ar[d]^-{\tilde{m}} \ar[l]_-{\sim} \\
      & H_{i+j-n}(X^3, p^*C_*(\Omega Y)) \ar[d]^-{p_*}\\
      H_{i+j-n}(\mathcal{P}_{X,f}Y) &  H_{i+j-n}(X^2, C_*(\Omega Y)) \ar[l]_-{\sim}.
    }
    $$

We know from Theorem \ref{cor : morphism of fibrations relative to ...} that the bottom square commutes. The top square commutes up to the sign $(-1)^{n-n(i+j)}$ and we now prove that the second square commutes up to the same sign.

For this purpose, we will now compare our orientation conventions to the one used in \cite{hingston2022product} and \cite{Max23}.

Denote $U_{\Delta} \subset X^2$ a tubular neighborhood of the diagonal $\Delta X$ in $X^2$ and denote $\tau_{\Delta} \in C^n(U_{\Delta}, U_{\Delta} \setminus \Delta X)$ a Thom class of the normal bundle $N_{\Delta} \cong TX$ such that $$\tau_{\Delta} \cap [X^2] = \Delta_*([X]).$$
This yields the "singular" orientation $\left(\Or \ N_{\Delta}, \Or \  T(\Delta X) \right)= \Or \ T(X \times X)\lvert_{\Delta X}$ which induces the orientation $\Or \ N_{\Delta} = (-1)^n \Or \ X.$

Denote $$U_{C^f} = (\ev_1 \times \ev_0)^{-1}U_{\Delta} \subset \mathcal{P}_{X,f} Y \times \mathcal{P}_{X,f} Y$$ and define the Thom class $\tau_{C^f} = (\ev_1 \times \ev_0)^*[\tau_{\Delta}] \in H^n(U_{C^f}, U_{C^f} \setminus (\ev_1 \times \ev_0)^{-1}(\Delta X)).$

Consider $p_{2,3} : X^4 \to X^2$, $p_{2,3}(a,b,c,d) = (b,c)$ and $U_D = p_{2,3}^{-1}(U_{\Delta})$ a tubular neighborhood of $D(X^3) \subset X^4$.
It is proven in \cite[Theorem 3.2]{Rie24} that the diagram

$$\xymatrix{
H_{i+j}(\mathcal{P}_{X,f}Y \times \mathcal{P}_{X,f}Y) \ar[d]^-{R_{C^f,*} \circ (\ev \times \ev)^*\tau^M_{D} \cap} & H_{i+j}(X^4,C_*(\Omega Y^2)) \ar[d]^-{D_!} \ar[l]_-{\sim} \\
H_{i+j-n}(\mathcal{P}_{X,f}Y \ftimes{\ev_1}{\ev_0} \mathcal{P}_{X,f}Y)  & H_{i+j-n}(X^3, D^*C_*(\Omega Y^2))  \ar[l]_-{\sim}
}$$

commutes up to the sign $(-1)^{n(i+j)+n}$ where $[\tau^M_D] \in H^n(U_D, U_D \setminus D(X^3))$ is the Thom class associated with the normal bundle $N_D \cong TX$ of the tubular neighborhood $U_D$ with the "Morse" orientation convention $$\left(\Or \ D(X^3), \Or \ N_{D}\right) = \Or \ X^4.$$ This yields the following equation.
$$[\tau_D^M] \cap [X^4] = (-1)^nD_*([X^3]).$$

It remains to prove that $$(\ev \times \ev)^*[\tau_D^M]  = (\ev_1 \times \ev_0)^*[\tau_{\Delta}] = [\tau_{C^f}] \in H^n(U_{C^f}, U_{C^f} \setminus (\ev \times \ev)^{-1}(DX^3)).$$

Consider $[\tau_D] = p_{2,3}^*[\tau_{\Delta}] \in H^n(U_D, U_D \setminus DX^3).$ Since $p_{2,3} \circ (\ev \times \ev) = \ev_1 \times \ev_0 : \mathcal{P}_{X,f}Y \times \mathcal{P}_{X,f} Y \to X^2$, it follows $$(\ev \times \ev)^*[\tau_D] = (\ev_1 \times \ev_0)^*[\tau_{\Delta}] = [\tau_{C^f}].$$

We now prove that $\tau_D = \tau_D^M$ by showing that $$[\tau_D] \cap [X^4] = (-1)^nD_*([X^3]).$$
Let $\sigma_1 \in C^n(X)$ be a representative of the fundamental class $[X]$ and $\tau_D \in C^n(U_D, U_D \setminus DX^3)$ be a representative of $[\tau_D]$. We will denote $\sigma_1=\sigma_2=\sigma_3=\sigma_4$.
\begin{align*}
    &\tau_D \cap (\sigma_1 \times \sigma_2 \times \sigma_3 \times \sigma_4)\\
    &= \sum_{\substack{J \sqcup K = \{1, \dots, 4n\} \\ |J| = n}} (-1)^{\textup{sgn}(J,K)} \tau_{D}(\sigma_1 \times \sigma_2 \times \sigma_3 \times \sigma_4)\lvert_{I^{J} \times \{0\}^K} (\sigma_1 \times \sigma_2 \times \sigma_3 \times \sigma_4)\lvert_{\{1\}^{J} \times I^K} \\
    &\overset{(1)}{=} \sum_{\substack{J \sqcup K = \{1, \dots, 4n\} \\ J \subset \{n+1, \dots, 3n\} \\ |J|=n}} (-1)^{\textup{sgn}(J,K)} \tau_{\Delta}(\sigma_2 \times \sigma_3)\lvert_{I^{J} \times \{0\}^K} (\sigma_1 \times \sigma_2 \times \sigma_3 \times \sigma_4)\lvert_{\{1\}^{J} \times I^K} \\
    &\overset{(2)}{=} \sum_{\substack{J' \sqcup K' = \{1, \dots, 2n\} \\ |J'|=n} } (-1)^{\textup{sgn}(J',K')} \tau_{\Delta}(\sigma_2 \times \sigma_3)\lvert_{I^{J'} \times \{0\}^{K'}} (\sigma_1 \times (\sigma_2 \times \sigma_3)\lvert_{\{1\}^{J'} \times I^{K'}} \times \sigma_4)\\
    &= (\Id \times \tau_{\Delta} \times \Id) \cap (\sigma_1 \times \sigma_2 \times \sigma_3 \times \sigma_4).
\end{align*}

    The equality $(1)$ is true since the chain $p_{2,3,*}(\sigma_1 \times \sigma_2 \times \sigma_3 \times \sigma_4)\lvert_{I^{J} \times \{0\}^K} = (\sigma_2 \times \sigma_3)\lvert_{I^{J} \times \{0\}^K}$ is degenerate and therefore equal to $0$ if $J \not \subset \{n+1, \dots, 3n\}$.
    
    The equality $(2)$ comes from the fact that the transformation $J \overset{\nu}{\to} J'$ is performed by the permutation $\nu= \begin{pmatrix}
        0 & 0 & I_n & 0\\
        I_n & 0 & 0 & 0 \\
        0 & I_n & 0 & 0 \\
        0 & 0 & 0 & I_n
    \end{pmatrix} $ that has signature $1$. Since $J \subset \{n+1, \dots, 3n\},$ then $\nu J \subset \{1, \dots, 2n\}$ and we consider $K'$ to be the complement of $J'=\nu J$ in $\{1, \dots, 2n\}$.
It follows that
    \begin{align*}
     [\tau_D] \cap [X^4] &= [\Id \times \tau_{\Delta} \times \Id] \cap [X^4] = (-1)^n (\Id \times \Delta \times \Id)_*[X^3] = (-1)^n D_*[X^3].  
    \end{align*}
    
    Therefore $[\tau_D] = [\tau_D^M]$ and $(\ev \times \ev)^*[\tau^M_D] = [\tau_{C^f}].$
\end{proof}

\section{Compatibility with the Chas-Sullivan product}\label{section : PPDG Compatibility with the Chas-Sullivan product}

In this section we suppose that $Y^k$ is a closed, oriented, connected manifold of dimension $k$ and we endow $H_{*}(Y,\COY)$ with the  ring structure induced by the product $$\CS : H_i(Y,\COY) \otimes H_j(Y,\COY) \to H_{i+j-n}(Y, \COY).$$
Let $f : X \to Y$ be a continuous map and let $I := I_f : X^2 \to Y \times X^2$, $I(x,x') = (f(x),x,x')$.
Define

$$\begin{array}{ccl}
    \cdot_M :H_i(Y, \COY) \otimes H_j(X^2, \COY)  & \overset{(-1)^{kj}K}{\longrightarrow} & H_{i+j-k}(Y \times X^2, C_*(\Omega Y^2)) \\
     & \overset{I_!}{\longrightarrow}  & H_{i+j-k}(X^2, I^*C_*(\Omega Y^2)) \\
     & \overset{\tilde{m}}{\longrightarrow} &  H_{i+j-k}(X^2, \COY),
\end{array}$$

where $m : \ls{Y} \ftimes{\ev}{\ev_0} \mathcal{P}_{X,f}Y \to \mathcal{P}_{X,f} Y$ is the morphism of fibrations over $X^2$ induced by the concatenation.

\begin{prop}\label{prop : module structure over Chas-Sullivan product}
    The pairing $\cdot_M : H_i(Y, \COY) \otimes H_j(X^2, \COY) \to H_{i+j}(X^2, \COY)$ induces a left $H_*(Y,\COY)$-module structure  on $H_*(X^2, \COY)$ with the additional property that if $\gamma \in H_*(Y, \COY)$ and $\tau, \tau' \in H_*(X^2, \COY)$, then 

    $$\gamma \cdot_M \PP(\tau \otimes \tau') = \PP((\gamma \cdot_M \tau) \otimes \tau').$$

\end{prop}

\begin{cor}
    This module structure induces a morphism of rings 

    $$\begin{array}{ccc}
       H_i(Y, \COY) &\to& H_{i+n-k}(X^2,\COY)    \\
         \gamma & \mapsto & \gamma \cdot_M 1_{PP}
    \end{array}.$$
\end{cor}

\begin{proof}[Proof of Proposition \ref{prop : module structure over Chas-Sullivan product}] 

We now prove that $\cdot_M$ is indeed a module structure, \emph{i.e} for all $\gamma,\gamma' \in H_*(Y,\Omega Y)$ and $\tau \in H_*(X^2,\Omega Y)$ 
$$\CS(\gamma \otimes \gamma') \cdot_M \tau = \gamma \cdot_M (\gamma' \cdot_M \tau).$$ 

Denote $\F^i = \COYi{i}$ for any $i  \in \N^*$ and $\F = \F^1 = \COY$
Let $\gamma,\gamma' \in H_*(Y,\F)$ and $\tau \in H_*(X^2,\F)$.

\begin{align*}
    \CS(\gamma \otimes \gamma') \cdot_M \tau &= (-1)^{k + k|\gamma'|} (-1)^{ k|\tau|} \tilde{m} \circ I_! \circ K \circ (\tilde{m} \otimes \Id) \circ (\Delta_! \otimes \Id) \circ (K \otimes \Id)(\gamma \otimes \gamma' \otimes \tau).
\end{align*}
\begin{align*}
    \gamma \cdot_M (\gamma' \cdot_M\tau) = (-1)^{ k|\tau|}(-1)^{k(|\gamma'| + |\tau|-k)}(-1)^{k|\gamma|} \tilde{m} \circ I_! \circ K \circ (\Id \otimes \tilde{m}) \circ (\Id \otimes I_!) \circ (\Id \otimes K)(\gamma \otimes \gamma' \otimes \tau).
\end{align*}

The equality $ \CS(\gamma \otimes \gamma') \cdot_M \tau = \gamma \cdot_M (\gamma' \cdot_M\tau)$ is hence equivalent to the commutativity of the following diagram up to the sign $(-1)^{k|\tau| + k|\gamma|}.$

$$\footnotesize \xymatrix{
 H_*(Y,\F)^{\otimes 2} \otimes H_*(X^2,\F) \ar[d]^{1 \otimes K} \ar[r]^{K \otimes 1} & H_*(Y^2, \F^{2}) \otimes H_*(X^2, \F) \ar[d]^K \ar[r]^{\Delta_! \otimes 1} \ar@{}[rd]|{\blue{(1)}}  & H_*(Y, \Delta^*\F^{2}) \otimes H_*(X^2, \F) \ar[d]^K \ar[r]^{\tilde{m} \otimes 1} & H_*(Y, \F) \otimes H_*(X^2, \F) \ar[d]^{K} \\
 H_*(Y, \F) \otimes H_*(Y \times X^2, \F^{2}) \ar@{}[rd]|{\blue{(2)}} \ar[d]^{1 \otimes I_!} \ar[r]^-K
 & H_*(Y^2 \times X^2, \F^{3}) \ar[r]^-{(\Delta \times 1)_!} \ar[d]^{(1 \times I)_!} & H_*(Y\times X^2, (\Delta\times 1)^*\F^{3}) \ar[r]^-{\widetilde{m \times 1}} \ar[d]^{I_!} & H_*(Y \times X^2, \F^{2}) \ar[d]^{I_!}\\
 H_*(Y,\F) \otimes H_*(X^2, I^*\F^{2}) \ar[d]^{1 \otimes \tilde{m}} \ar[r]^K
  &  H_*(Y \times X^2, (1 \times I)^* \F^{3}) \ar[d]^-{\widetilde{1 \times m}} \ar[r]^-{I_!}  & H_*(X^2,I^*(\Delta\times 1)^*\F^{3}) \ar[d]^-{I^*\widetilde{1 \times m}} \ar[r]^-{I^*\widetilde{m \times 1}}& H_*(X^2, I^*\F^{2}) \ar[d]^{\tilde{m}} \\
  H_*(Y, \F) \otimes H_*(X^2, \F) \ar[r]^K  & H_*(Y \times X^2, \F^{2}) \ar[r]^-{I_!}  & H_*(X^2, I^* \F^{2}) \ar[r]^{\tilde{m}}   & H_*(X^2,\F). 
}$$

Indeed, it is shown in \cite[Lemma 6.17]{Rie24} the squares \blue{(1)} and \blue{(2)} commute up to the sign $(-1)^{k|\tau|}$ and  $(-1)^{k|\gamma|}$ respectively.
We prove that $$\gamma \cdot_M \PP(\tau \otimes \tau') = \PP((\gamma \cdot_M \tau) \otimes \tau')$$ using similar arguments.
This equality is equivalent to the commutativity of the following diagram up to the sign $(-1)^{k|\tau'|} (-1)^{kn} (-1)^{n|\gamma|}.$

$$\footnotesize \xymatrix{
    H_*(Y, \F) \otimes H_*(X^2,\F)^{\otimes 2} \ar[d]^{1 \otimes K} \ar[r]^-{K \otimes 1} & H_*(Y\times X^2, \F^2) \otimes H_*(X^2, \F) \ar[d]^K \ar[r]^-{I_! \otimes 1} \ar@{}[rd]|{\blue{(1)}} & H_*(X^2, I^*\F^2) \otimes H_*(X^2, \F) \ar[d]^K \ar[r]^-{\tilde{m} \otimes 1} &  H_*(X^2,\F)^{\otimes 2} \ar[d]^K  \\
    H_*(Y,\F) \otimes H_*(X^4, \F^2) \ar[d]^{1 \otimes D_!} \ar[r]^-{K} \ar@{}[rd]|{\blue{(2)}} & H_*(Y \times X^4, \F^3) \ar[d]^{(1 \times D)_!} \ar[r]^-{(I \times 1)_!} & H_*(X^4, (I \times 1)^*\F^3) \ar[d]^{D_!} \ar[r]^-{\widetilde{m \times 1}} & H_*(X^4,\F^2) \ar[d]^{D_!}  \\
    H_*(Y,\F) \otimes H_*(X^3, D^*\F^2) \ar[d]^{1 \otimes \tilde{m}} \ar[r]^-{K}     & H_*(Y \times X^3, (1 \times D)^*\F^3) \ar[d]^{\widetilde{1 \times m}} \ar[r]^-{I_!}  & H_*(X^3, (I \times 1)^*D^*\F^3) \ar[d]^{\widetilde{1 \times m}} \ar[r]^-{D^*\widetilde{m \times 1}} &  H_*(X^3, D^*\F^2) \ar[d]^{\tilde{m}} \\
    H_*(Y, \F) \otimes H_*(X^3, p^*\F) \ar[d]^{1 \otimes p_*} \ar[r]^-{K} & H_*(Y \times X^3, (1 \times p)^*\F^3) \ar[d]^{(1 \times p)_*} \ar[r]^-{(I \times 1)_!} \ar@{}[rd]|{\blue{(3)}} & H_*(X^3, p^*I^*\F^2) \ar[d]^{p_*} \ar[r]^-{p^*\tilde{m}} & H_*(X^3,p^*\F) \ar[d]^{p_*}   \\
    H_*(Y, \F) \otimes H_*(X^2, \F) \ar[r]^-K &  H_*(Y \times X^2, \F^2) \ar[r]^-{I_!} & H_*(X^2, I^*\F^2) \ar[r]^-{\tilde{m}} & H_*(X^2,\F) 
    }$$

It is shown in \cite[Lemma 6.17]{Rie24} the squares \blue{(1)} and \blue{(2)} commute up to the sign $(-1)^{k|\tau'|}$ and  $(-1)^{k|\gamma|}$ respectively. We now prove that square \blue{(3)} commutes up to the sign $(-1)^{kn}$ using Lemma \ref{lemme : commutativity shriek and direct}.

First, we check that $(\Id \times p)  \circ (I \times \Id)=  I \circ p : X^3 \to Y \times X^2$, $(a,b,c) \mapsto (f(a),a,c)$ and 
$$(\Id \times p)^{-1}(I(X^2)) = \{(f(a),a,b,c), \ (a,b,c) \in X^3\} = (I \times \Id)(X^3).$$

The normal bundles $N_I \cong TY \cong N_{I \times \Id}$ are isomorphic and their orientations are given by the rules

$$(\Or \ X^2, \Or \ N_I) = (\Or \ Y, \Or \ X^2)$$
and $$(\Or \ X^3, \Or \ N_{I \times \Id}) = (\Or \ Y, \Or \ X^3).$$

We compute
\begin{align*}
    (\Or \ X^3, \Or \ N_{I}) &= (-1)^{nk}(\Or \ X^2, \Or \ N_I, \Or \ X)\\
    &= (-1)^{nk} (\Or \ Y, \Or \ X^2, \Or \ X) \\
    &= (-1)^{nk} (\Or \ Y, \Or \ X^3).
\end{align*}

Therefore $\Or \ N_I = (-1)^{nk} \Or \ N_{I \times \Id}.$ It follows that the square \blue{(3)} commutes up to the sign $(-1)^{nk}.$

\end{proof}

\section{Homotopy invariance}\label{section : PPDG homotopy invariance}

We prove here that the ring and module structure on $H_*(X^2, C_*(\Omega Y))$ only depends on the homotopy type of $f : X \to Y.$ More precisely, if $g : X \to Y$ is a continuous map homotopic to $f$, there exists a chain homotopy equivalence $$  \underbrace{C_*(X^2, \Xi,(f \times f)^*\COY)}_{\simeq C_*(\mathcal{P}_{X,f} Y)} \to \underbrace{C_*(X^2, \Xi,(g\times g)^*C_*(\Omega Y))}_{ \simeq C_*(\mathcal{P}_{X,g} Y)}$$   that induces  an isomorphism of rings and modules at the homology level. We recall the notations $$C_*(X, \Xi, \COY ; f) := C_*(X^2, \Xi,  (f \times f)^*\COY) $$ and $$C_*(X, \Xi, \COY ; g) := C_*(X^2, \Xi,  (g \times g)^*\COY).$$

We will also prove that if $X_1 \subset Y_1$ and $X_2 \subset Y_2$ are manifolds included in topological spaces $Y_1, Y_2$ such that there exists a homotopy equivalence of pairs $\varphi : (Y_1,X_1) \to (Y_2, X_2)$, then there exists an isomorphism of rings $$H_*((X_1)^2, C_*(\Omega Y_1)) \overset{\sim}{\to} H_*((X_2)^2, C_*(\Omega Y_2)).$$ 

\subsection{DG Identification morphism in the case of fibrations}

Let $(X,\star)$ and $(Y, \star_Y)$ be pointed, oriented, closed, and connected manifolds. Let $F \hookrightarrow E \overset{\pi}{\to} Y$ be a fibration and $\Phi : E \ftimes{\pi}{ev_0} \mathcal{P}Y \to E$ be a transitive lifting function associated with this fibration.
    Let $\varphi_0 ,\varphi_1 : X \to Y$ be two homotopic maps such that $\varphi_0(\star) = \varphi_1(\star) = \star_Y.$ Let $\varphi : X \times [0,1] \to Y$ be a homotopy between $\varphi_0 = \varphi( \cdot, 0)$ and $\varphi_1 = \varphi(\cdot, 1).$ Denote 
    $$\deffct{\eta_{\varphi}}{X}{\mathcal{P}Y}{x}{(s \mapsto \varphi(x,s))}$$ the map that sends $x \in X$ to the path going from $\varphi_0(x)$ to $\varphi_1(x)$ defined by the homotopy $\varphi.$

\begin{prop}\label{prop : Identification morphism for fibrations}
    Let $\Xi = (f,\xi, \dots)$ be a set of DG Morse data on $X$. 
    The morphism of fibrations $$\deffct{\Psi^{\varphi}}{\varphi_0^*E}{\varphi_1^*E}{(x,e)}{\left(x,\Phi(e,\eta_{\varphi}(x))\right)}$$ induces a chain homotopy equivalence $$\tilde{\Psi}^{\varphi} : C_*(X,\Xi,\varphi_0^*C_*(F)) \to C_*(X,\Xi, \varphi_1^*C_*(F))$$ that satisfies the same properties as the Identification morphism defined in \cite[Proposition 8.2.1]{BDHO23}. More precisely, it holds that:
    \begin{enumerate}
        \item \label{item : 1} If $\varphi=\Id$ is the constant homotopy at $\varphi_0$, then $\tilde{\Psi}^{\Id} = \Id.$
        \item \label{item : 2} If two homotopies $\varphi$ and $\varphi'$ are homotopic with fixed endpoints, then $\tilde{\Psi}^{\varphi}$ and $\tilde{\Psi}^{\varphi'}$ are chain homotopic.
        \item \label{item : 3}If $\varphi_{01}$ is a homotopy between $\varphi_0$ and $\varphi_1$ and $\varphi_{12}$ is a homotopy between $\varphi_1$ and $\varphi_2$, then denoting $\varphi_{02}$ the concatenation of $\varphi_{01}$ and $\varphi_{12}$ we have that $\tilde{\Psi}^{\varphi_{12}} \circ \tilde{\Psi}^{\varphi_{01}}$ and $\tilde{\Psi}^{\varphi_{02}}$ are homotopic. Therefore,
        $$\tilde{\Psi}^{\varphi_{12}} \circ \tilde{\Psi}^{\varphi_{01}}=\tilde{\Psi}^{\varphi_{02}}$$ in homology. In particular $\tilde{\Psi}^{\varphi}$ is always a chain homotopy equivalence of chain homotopy inverse $\tilde{\Psi}^{\Bar{\varphi}}$.
        \item \label{item : 4} The maps $\tilde{\Psi}^{\varphi} \circ \varphi_{0,!}$ and $\varphi_{1,!}$ are homotopic. 
        \item \label{item : 5}The maps $\varphi_{1,*} \circ \tilde{\Psi}^{\varphi} $and $\varphi_{0,*}$ are homotopic.
    \end{enumerate}
\end{prop}

\begin{proof}
    \textbf{Proof of \ref{item : 1}:} If $b=\eta_{\Id}(x)$ is the constant path at $\varphi_0(x) = \pi(e)$, then $\Phi(e,b) =e$. It follows that  $\Psi^{\Id} = \Id$ and  $\tilde{\Psi}^{\Id} = \Id.$

    \textbf{Proof of \ref{item : 2}:} Let $H : X \times [0,1]^2 \to Y$ be a homotopy with fixed endpoints between $H(\cdot,\cdot, 0) = \varphi$ and $H(\cdot,\cdot, 1) = \varphi'$. For $s \in [0,1]$, denote $H(\cdot, \cdot,s) = : H_s$.
    Consider the family of morphism of fibrations $$\deffct{\Psi^{H_s}}{\varphi^*_0E}{\varphi_1^* E}{(x,e)}{(x,\Phi(e,\eta_{H_s}(x)))}$$ for $s \in [0,1].$ This is well-defined since $H_s(\cdot,0) = \varphi_0 = \varphi'_0$ and $H_s(\cdot,1) = \varphi_1 = \varphi'_1$ for all $s \in [0,1].$ For each $s \in [0,1]$, the morphism of fibrations $\Psi^{H_s}$ induces the morphism of complexes $$\tilde{\Psi}^{H_s} = \sum_n (\Psi^{H_s}_{n+1}\otimes 1)\m^n : C_k(X,\Xi,\varphi^*_0 C_*(F)) \to \bigoplus_{i+j = k} C_i(F) \otimes \Z \Crit_j(f).$$ Therefore, the map $$\deffct{\tilde{\Psi}^H} {C_k(X,\Xi,\varphi^*_0 C_*(F))}{\displaystyle\bigoplus_{i+j = k} C_{i+1}(F) \otimes \Z \Crit_j(f)}{\alpha \otimes x}{\displaystyle s \mapsto \sum_n (\Psi^{H_s}_{n+1} \otimes 1) \m^n(\alpha \otimes x)} $$ satisfies $$\partial \tilde{\Psi}^{H} + \tilde{\Psi}^H \partial = \tilde{\Psi}^{\varphi'} - \tilde{\Psi}^{\varphi}.$$

    \textbf{Proof of \ref{item : 3}:} This is a direct application of \cite[Corollary 5.18]{Rie24}. Indeed, since \begin{align*}
        \Psi^{\varphi_{12}}\circ \Psi^{\varphi_{01}}(x,e) &= (x, \Phi(\Phi(e,\eta_{\varphi_{01}}(x)),\eta_{\varphi_{12}}(x))) \\
        &= (x, \Phi(e, \eta_{\varphi_{01}}(x) \# \eta_{\varphi_{12}}(x))) \\
        &= (x, \Phi(e, \eta_{\varphi_{02}}(x)))\\
        &= \Psi^{\varphi_{02}}(x,e),
    \end{align*}
    
    it follows that $\tilde{\Psi}^{\varphi_{12}}\circ \tilde{\Psi}^{\varphi_{01}} = \tilde{\Psi}^{\varphi_{02}}.$ 

    \textbf{Proof of \ref{item : 4}:} For $s \in [0,1]$, define $$\deffct{\Psi^{s1}}{\varphi_s^*E}{\varphi_1^* E}{(x,e)}{(x, \Phi(e, \eta_{\varphi}(x)_{\lvert_{[s,1]}}))}$$ and consider $$ \deffct{\kappa}{C_*(Y,\Xi_Y, C_*(F))}{C_{*-k+1}(X,\Xi, \varphi_1^* C_*(F))}{\alpha \otimes x}{s \mapsto \tilde{\Psi}^{s1} \circ \varphi_{s,!}} $$

    Then, $$\partial \kappa + \kappa \partial = \varphi_{1,!} - \tilde{\Psi}^{\varphi} \circ \varphi_{0,!}. $$

    The proof of \ref{item : 5} is the same considering $$\deffct{\kappa'}{C_*(X,\Xi, \varphi_1^*C_*(F))}{C_{*+1}(Y,\Xi_Y, C_*(F))}{\beta \otimes y}{s \mapsto \varphi_{s,*} \circ \tilde{\Psi}^{0s},}$$

    where $$\deffct{\Psi^{0s}}{\varphi_0^*E}{\varphi_s^* E}{(x,e)}{(x, \Phi(e, \eta_{\varphi}(x)_{\lvert_{[0,s]}})).}$$

\end{proof}

\subsection{Invariance with regard to the homotopy type of \texorpdfstring{$f : X \to Y$}{f}}

Denote $\varphi_{f,g} : X \times [0,1] \to Y$ a homotopy between $\varphi_{f,g}(\cdot ,0) = f $ and $\varphi_{f,g}(\cdot,1) = g$. Denote $$\varphi^2_{f,g} : X^2 \times [0,1] \to Y^2, \ \varphi_{f,g}^2(a,b,t) = (\varphi_{f,g}(a,t), \varphi_{f,g}(b,t))$$ and $\varphi^2_{g,f} = \overline{\varphi^2_{f,g}} : X^2 \times [0,1] \to Y^2$. For readability, if $a \in X$, we will use the notation $$\eta(a) = (s\mapsto \varphi_{f,g}(a,s)) \in \mathcal{P}_{f(a) \to g(a)}Y.$$
Proposition \ref{prop : Identification morphism for fibrations} states that the morphism of fibrations $$\deffct{\Psi_{f,g} := \Psi^{\varphi^2_{f,g}}}{\mathcal{P}_{X,f} Y}{\mathcal{P}_{X,g} Y}{(x_0,x_1,\gamma)}{(x_0,x_1, \eta(x_0)^{-1}\gamma \eta(x_1))}$$ induces the chain homotopy equivalence 
$$\tilde{\Psi}_{f,g} : C_*(X, \Xi, C_*(\Omega Y) ; f) \to C_*(X, \Xi, C_*(\Omega Y) ;g)$$ with chain homotopy inverse $$\tilde{\Psi}_{g,f} := \widetilde{\Psi^{\varphi^2_{g,f}}} : C_*(X, \Xi, C_*(\Omega Y) ; g) \to C_*(X, \Xi, C_*(\Omega Y) ;f).$$

\begin{prop}\label{prop : ring iso homotopy type of f}
    The isomorphism $$\tilde{\Psi}_{f,g} : H_*(X, \COY ; f) \overset{\sim}{\to} H_*(X, \COY ; g)$$  is an isomorphism of rings.
\end{prop}

\begin{proof}  
    
    To prove that it is an isomorphism of rings, we show that the following diagram commutes:

    $$\xymatrix@C=5em{
    H_*(X^2, \COY ;f) ^{\otimes 2} \ar[d]^K \ar[r]^-{\tilde{\varphi}_{f,g}^{\otimes 2}} &  H_*(X^2, \COY ;g) ^{\otimes 2} \ar[d]^K  \\
    H_*(X^4, \COYi{2} ; f) \ar[r]^{\widetilde{\varphi_{f,g} \times \varphi_{f,g}}} \ar[d]^{D_!} & H_*(X^4, \COYi{2} ; g) \ar[d]^{D_!} \\
    H_*(X^3, D^*\COYi{2} ; f) \ar[r]^{D^*\widetilde{\varphi_{f,g} \times \varphi_{f,g}}} \ar[d]^{\tilde{m}} & H_*(X^3, D^*\COYi{2} ; g) \ar[d]^{\tilde{m}} \\
    H_*(X^3, p^*\COY ; f) \ar[r]^{p^* \tilde{\varphi}_{f,g}} \ar[d]^{p_*} &  H_*(X^3, p^*\COY ; g) \ar[d]^{p_*} \\
    H_*(X^2, \COY, f) \ar[r]^{\tilde{\varphi}_{f,g}} & H_*(X^2,\COY, g).
    }$$

All the squares commute by basic properties about the maps induced by morphisms of fibrations. We only have to check that, in homology,
$$m_* \circ (\varphi_{f,g} \times \varphi_{f,g})_* = \varphi_{f,g,*} \circ m_* : H_*(\mathcal{P}_{X,f} \ftimes{ev_1}{ev_0} \mathcal{P}_{X,f}) \to H_*(\mathcal{P}_{X,g}).$$
Indeed, the maps 
$$\deffct{m \circ \varphi_{f,g} \times \varphi_{f,g}}{\mathcal{P}_{X,f} \ftimes{ev_1}{ev_0} \mathcal{P}_{X,f}}{\mathcal{P}_{X,g}}{(x_0,x_1,x_2,\gamma_1,\gamma_2)}{(x_0,x_2, \eta(x_0)^{-1} \gamma_1 \eta(x_1) \eta(x_1)^{-1} \gamma_2 \eta(x_2))}$$ and 
$$\deffct{\varphi_{f,g} \circ m}{\mathcal{P}_{X,f} \ftimes{ev_1}{ev_0} \mathcal{P}_{X,f}}{\mathcal{P}_{X,g}}{(x_0,x_1,x_2,\gamma_1,\gamma_2)}{(x_0,x_2,\eta(x_0)^{-1} \gamma_1\gamma_2 \eta(x_2)),}$$ are homotopic.
This concludes that  $\tilde{\Psi}_{f,g} : H_*(X, \COY ; f) \overset{\sim}{\to} H_*(X, \COY ; g)$ is an isomorphism of rings.
\end{proof}

\begin{prop}\label{prop : module iso homotopy type of f}
The isomorphism $$\tilde{\Psi}_{f,g} : H_*(X, \COY ; f) \overset{\sim}{\to} H_*(X, \COY ; g)$$  is an isomorphism of $H_*(Y, \COY)$-modules.
\end{prop}

\begin{proof}
We prove that the following diagram commutes:

$$
\xymatrix{
H_*(Y, C_*(\Omega Y)) \otimes H_*(X^2,C_*(\Omega Y) ;f) \ar[d]^-K \ar[r]^{1 \otimes \tilde{\Psi}_{f,g}} & H_*(Y, C_*(\Omega Y)) \otimes H_*(X^2,C_*(\Omega Y) ;g) \ar[d]^-K & \\
H_*(Y \times X^2, C_*(\Omega Y^2) ; f) \ar[d]^{I_{f,!}} \ar[r]^{\widetilde{1 \times \Psi_{f,g}}} & H_*(Y \times X^2, C_*(\Omega Y^2) ; g) \ar[d]^{I_{f,!}} \ar[dr]^{I_{g,!}} & \\
H_*(X^2, I_f^* C_*(\Omega Y^2) ;f) \ar[d]^{\tilde{m}} \ar[r]^{I_{f}^*\widetilde{1 \times \Psi_{f,g}}} &  H_*(X^2, I_f^* C_*(\Omega Y^2) ;g)  \ar[r]^{\tilde{\Psi}^{H}} & H_*(X^2, I_{g}^*C_*(\Omega Y) ;g) \ar[dl]^{\tilde{m}} \\
H_*(X^2,C_*(\Omega Y) ;f) \ar[r]^{\tilde{\Psi}_{f,g}} & H_*(X^2,C_*(\Omega Y) ;g) & 
}
$$

    where $H : X^2 \times [0,1] \to Y \times X^2$, $H(a,b,t) = (\varphi_{f,g}(a,t), a, b)$ is a homotopy between $I_f : X^2 \to Y \times X^2$, $I_f(a,b) = (f(a),a,b)$ and $I_g : X^2 \to Y \times X^2$, $I_g(a,b) = (g(a),a,b).$ Therefore $$\deffct{\Psi^H}{\ls{Y} \ftimes{\ev}{f(\ev_0)}\mathcal{P}_{X,g} Y}{\ls{Y} \ftimes{\ev}{g(\ev_0)}\mathcal{P}_{X,g} Y}{(\gamma,(x_0,x_1,\tau))}{\left(\eta(x_0)^{-1}\gamma \eta(x_0), (x_0,x_1,\tau)\right)}.$$

It follows from compatibility properties between morphisms of fibrations, shriek maps and $K$ that the two squares on the top commute. Proposition \ref{prop : Identification morphism for fibrations} 4. states that the triangle on the right commutes.
Since 
$$m \circ \Psi^H\circ(1 \times \Psi_{f,g}) : \ls{Y} \ftimes{\ev}{f(\ev_0)} \mathcal{P}_{X,f}Y \to \ls{Y} \ftimes{\ev}{f(\ev_0)} \mathcal{P}_{X,g}Y \to \ls{Y} \ftimes{\ev}{g(\ev_0)} \mathcal{P}_{X,g} Y \to  \mathcal{P}_{X,g} Y$$
$$m \circ \Psi^H\circ(1 \times \Psi_{f,g})(\gamma,(x_0,x_1,\tau)) = (x_0,x_1, \eta(x_0)^{-1}\gamma \eta(x_0)\eta(x_0)^{-1}\tau \eta(x_1))$$ is homotopic to 
$$\Psi_{f,g} \circ m : \ls{Y} \ftimes{\ev}{f(\ev_0)} \mathcal{P}_{X,f}Y \to \mathcal{P}_{X,f}Y \to \mathcal{P}_{X,g}Y $$
$$\Psi_{f,g} \circ m (\gamma,(x_0,x_1,\tau)) = (x_0,x_1,\eta(x_0)^{-1} \gamma \tau \eta(x_1)),$$

it follows that $\tilde{m} \circ \tilde{\Psi}^H\circ \widetilde{(1 \times \Psi_{f,g})} = \tilde{\Psi}_{f,g} \circ \tilde{m} : H_*(X^2, I_f^*C_*(\Omega Y^2);f) \to H_*(X^2,C_*(\Omega Y) ;g).$

This concludes that the diagram commutes and that $\tilde{\Psi}_{f,g} : H_*(X, \COY ; f) \overset{\sim}{\to} H_*(X, \COY ; g)$ is an isomorphism of $H_*(Y , C_*(\Omega Y))$-modules.

\end{proof}

\subsection{Homotopy invariance of the product}

Let $Y_1,Y_2$ be two topological spaces and $X_1 \overset{i_1}{\subset} Y_1$, $X_2 \overset{i_2}{\subset} Y_2$ be two pointed, oriented, closed, and connected manifolds. Assume that there exists a homotopy equivalence of pairs $\varphi : (Y_1, X_1) \to (Y_2, X_2)$ and denote $\psi : (Y_2, X_2) \to (Y_1, X_1)$ a homotopy inverse of $\varphi.$ Denote $n= \dim(X_1)=\dim(X_2)$ and $d = \deg(\varphi \lvert_{X_1}) \in \{-1,1\}.$

We will denote $\mathcal{P}\varphi : \mathcal{P}_{X_1} Y_1 \to \varphi^{2,*}\mathcal{P}_{X_2} Y_2 $ the morphism of fibrations over $X_1^2$ defined by $$\mathcal{P}\varphi(x_0,x_1,\gamma) = (x_0,x_1, \varphi \circ \gamma).$$

The map $\mathcal{P}\varphi $ is a homotopy equivalence and therefore induces an isomorphism $$\widetilde{\mathcal{P}\varphi} : H_*((X_1)^2,C_*(\Omega Y_1)) \to H_*((X_1)^2, \varphi^{2,*}C_*(\Omega Y_2)).$$

\begin{prop}\label{prop : ring iso homotopy equivalence}
    The map $$\varphi_{\#} = d \cdot (\widetilde{\mathcal{P}\varphi})^{-1} \circ (\varphi^2)_! : H_*(X_2^2, C_*(\Omega Y_2)) \to H_*(X_1^2, C_*(\Omega Y_1)) $$ is an isomorphism of rings.
\end{prop}

\begin{proof}
    It suffices to prove that the following diagram commutes up to the sign $d$.

$$
\xymatrix@C=35pt{
H_*((X_2)^2, C_*(\Omega Y_2))^{\otimes 2} \ar[r]^-{(\varphi^2)_!^{\otimes 2}} \ar[d]^-{K} & H_*((X_1)^2, \varphi^{2,*}C_*(\Omega Y_2))^{\otimes 2}  \ar[d]^-{K} & H_*((X_1)^2,C_*(\Omega Y_1)) \ar[d]^-K \ar[l]_-{\widetilde{\mathcal{P}\varphi}^{\otimes 2}}\\
H_*((X_2)^4,C_*(\Omega Y_2^2)) \ar[d]^-{D_{2,!}} \ar[r]^-{(\varphi^4)_!} & H_*((X_1)^4, \varphi^{4,*}C_*(\Omega Y_2^2)) \ar[d]^-{D_{1,!}}  & H_*((X_1)^4, C_*(\Omega Y_1^2)) \ar[d]^-{D_{1,!}} \ar[l]_-{\widetilde{\mathcal{P}\varphi \times \mathcal{P}\varphi}} \\
H_*((X_2)^3, D_2^* C_*(\Omega Y_2^2)) \ar[r]^-{(\varphi^3)_!} \ar[d]^-{\tilde{m}_2}& H_*((X_1)^3, \varphi^{3,*} D_1^* C_*(\Omega Y_2^2)) \ar[d]^-{\varphi^{3,*}\tilde{m}_2} & H_*((X_1)^3, D_1^*C_*(\Omega Y_1^2)) \ar[l]_-{D_1^* \widetilde{\mathcal{P}\varphi \times \mathcal{P}\varphi}} \ar[d]^-{\tilde{m}_1} \\
H_*((X_2)^3, p_2^*C_*(\Omega Y_2)) \ar[r]^-{(\varphi^3)_!} \ar@{}[rd]|{\blue{(*)}} \ar[d]^-{p_{2,*}} & H_*((X_1)^3, p_1^*\varphi^{2,*}C_*(\Omega X_2)) \ar[d]^-{p_{1,*}} & H_*((X_1)^3, p_1^*C_*(\Omega Y_1)) \ar[l]_-{p_{1}^*\widetilde{\mathcal{P}\varphi}} \ar[d]^-{p_{1,*}} \\
H_*((X_2)^2, C_*(\Omega Y_1)) \ar[r]^-{(\varphi^2)_!} & H_*((X_1)^2, \varphi^{2,*} C_*(\Omega Y_2)) & H_*((X_1)^2, C_*(\Omega Y_1)) \ar[l]_-{\widetilde{\mathcal{P}\varphi}}.
}
$$

The square $\blue{(*)}$ commutes up to the sign $d$ since $(\varphi^3)_! = d^3 (\varphi^3)_*^{-1} = d (\varphi^{3})_*^{-1}$ and $(\varphi^2)_! = d^2 (\varphi^{2})_*^{-1}= (\varphi^{2})_*^{-1} $.

\end{proof}

\begin{prop}\label{prop : module iso homotopy equivalence}
    The isomorphism $\varphi_{\#} : H_*((X_2)^2, C_*(\Omega Y_2)) \to H_*((X_1)^2, C_*(\Omega Y_1))$ satisfies the equation
    $$\forall \gamma \in H_*(Y_2, C_*(\Omega Y_2)), \forall \tau \in H_*((X_2)^2, C_*(\Omega Y_2)), \ \varphi_{\#}(\gamma \cdot_M \tau) = \varphi^{\mathcal{L}}_{\#}(\gamma) \cdot_M \varphi_{\#}(\tau),$$ 
    where $\varphi^{\mathcal{L}}_{\#} = \widetilde{\mathcal{L}\varphi}^{-1} \circ \varphi_! : H_*(Y_2, C_*(\Omega Y_2)) \to H_*(Y_1, C_*(\Omega Y_1))$ is the ring isomorphism induced by the homotopy equivalence $\varphi : Y_1 \to Y_2$ (see \cite[Proposition 7.13]{Rie24}).
\end{prop}
\begin{proof}
Denote $\F_i^j = C_*(\Omega Y_i^j)$ and $\F_i = \F_i^1$ for $i \in \{1,2\}$ and $j \geq 0.$ The following diagram commutes and concludes the proof.
    $$\xymatrix{
        H_*(Y_2, \F_2) \otimes H_*((X_2)^2,\F_2) \ar[r]^-{\varphi_! \otimes \varphi^2_!} \ar[d]^K & H_*(Y_1, \varphi^*\F_2) \otimes H_*((X_1)^2, \varphi^{2,*} \F_2) \ar[d]^K &  H_*(Y_1, \F_1) \otimes H_*((X_1)^2, \F_1) \ar[d]^K \ar[l]_-{\widetilde{\mathcal{L}\varphi} \otimes \widetilde{\mathcal{P}\varphi}} \\
        H_*(Y_2\times (X_2)^2, \F_2^2) \ar[r]^-{ \varphi^3_!}  \ar[d]^-{I_{2,!}} &  H_*(Y_1 \times (X_1)^2, \varphi^{3,*} \F_2^2)  \ar[d]^-{I_{1,!}} &   H_*(Y_1\times (X_1)^2, \F_1^2) \ar[l]_-{\widetilde{\mathcal{L} \varphi \times \mathcal{P}\varphi}}  \ar[d]^-{I_{1,!}} \\
        H_*((X_2)^2, I_2^*\F_2^2) \ar[r]^-{\varphi^3_!} \ar[d]^-{\tilde{m}} &  H_*((X_1)^2, I_1^*\varphi^{3,*}\F_1^2) \ar[d]^-{\varphi^{2,*}\tilde{m}}& H_*((X_1)^2, I_2^*\F_1^2) \ar[l]_-{I_1^*\widetilde{\mathcal{L} \varphi \times \mathcal{P}\varphi}} \ar[d]^-{\tilde{m}}\\
        H_*((X_2)^2, \F_2) \ar[r]^-{\varphi^2_!} & H_*((X_1)^2, \varphi^{2,*}\F_2) & H_*((X_1)^2, \F_1) \ar[l]_-{\widetilde{\mathcal{P}\varphi}}
    }$$
\end{proof}

\setcounter{secnumdepth}{-1}

\bibliography{mybib}
\bibliographystyle{alpha}

\end{document}